\documentclass[12pt]{article}

\usepackage[top=2.5cm,left=2.5cm,right=2.5cm,bottom=2.5cm]{geometry}

\usepackage{algorithmic,algorithm}
\usepackage{amsmath, graphicx,epsfig ,enumerate,latexsym,amssymb}

\newtheorem{theorem}{Theorem}

\newtheorem{proposition}{Proposition}
\newtheorem{lemma}{Lemma}
\newtheorem{corollary}{Corollary}

\newcommand{\ignore}[1]{}
\newcommand{\qed}{\hfill$\rule{2mm}{3mm}$}
\newenvironment{proof}{\par{\noindent \bf Proof:}}{\qed \par}

\newcommand{\E}{\mathrm{I\!E}}
\renewcommand{\P}{\mathrm{I\!P}}

\newsavebox{\traitbox}

\title{Convergence Speed of Binary Interval Consensus\thanks{A preliminary version without proofs of this work first appeared in \cite{DV10}}}

\author{Moez Draief\thanks{Imperial College London, SW7 2AZ London, UK (m.draief@imperial.ac.uk). } \and Milan Vojnovi\' c\thanks{Microsoft Research, J.J. Thomson Avenue, CB3 0FB Cambridge, UK, ( milanv@microsoft.com)}}

\begin{document}

\maketitle

\begin{abstract} We consider the convergence time for solving the binary consensus problem using the interval consensus algorithm proposed by B\' en\' ezit, Thiran and Vetterli (2009). In the binary consensus problem, each node initially holds one of two states and the goal for each node is to correctly decide which one of these two states was initially held by a majority of nodes. 

We derive an upper bound on the expected convergence time that holds for arbitrary connected graphs, which is based on the location of eigenvalues of some contact rate matrices. We instantiate our bound for particular networks of interest, including complete graphs, paths, cycles, star-shaped networks, and Erd\" os-R\' enyi random graphs; for these graphs, we  compare our bound with alternative computations. We find that for all these examples our bound is tight, yielding the exact order with respect to the number of nodes.

We pinpoint the fact that the expected convergence time critically depends on the voting margin defined as the difference between the fraction of nodes that initially held the majority and the minority states, respectively. The characterization of the expected convergence time yields exact relation between the expected convergence time and the voting margin, for some of these graphs, which reveals how the expected convergence time goes to infinity as the voting margin approaches zero.  

Our results provide insights into how the expected convergence time depends on the network topology which can be used for performance evaluation and network design. The results are of interest in the context of networked systems, in particular, peer-to-peer networks, sensor networks and distributed databases. 
\end{abstract}

\maketitle

\section{Introduction}

Algorithms for distributed computation in networks have recently attracted considerable interest because of their wide-range of applications in networked systems such as peer-to-peer networks, sensor networks, distributed databases, and on-line social networks. A specific algorithmic problem of interest is the so called \emph{binary consensus}~\cite{srikant,perron,benezit,nedic} where, initially, each node in the network holds one of two states and the goal for each node is to correctly decide which one of the two states was initially held by a majority of nodes. This is to be achieved by a decentralized algorithm where each node maintains its state based on the information exchanged at contacts with other nodes, where the contacts are restricted by the network topology. It is desired to reach a final decision by all nodes that is correct and within small convergence time. 

A typical application scenario of the binary consensus corresponds to a set of agents who want to reach consensus on whether a given event has occurred based on their individual, one-off collected, information. Such cooperative decision-making settings arise in a number of applications such as environmental monitoring, surveillance and security, and target tracking~\cite{mostofi08}, as well as voting in distributed systems \cite{Hardekopf}. Furthermore, it has been noted that one can use multiple binary consensus instances to solve multivalued consensuses; we refer to \cite{mostofi00,Zhang09} for an account on such algorithms. 

We consider a decentralized algorithm known as \emph{interval or quantized consensus} proposed by B\' en\' ezit, Thiran, and Vetterli~\cite{benezit}. The aim of this algorithm is to decide which one of $k\geq 2$ partitions of an interval contains the average of the initial values held by individual nodes. In this paper, we focus on \emph{binary} interval consensus, i.e. the case $k=2$. An attractive feature of the interval consensus is its accuracy; it was shown in~\cite{benezit} that for any finite connected graph that describes the network topology, the interval consensus is guaranteed to converge to the correct state with probability $1$. However, the following important question remained open: \emph{How fast does the interval consensus converge to the final state?} We answer this question for the case of binary interval consensus.

The interval consensus could be considered a state-of-the-art algorithm for solving the binary consensus problem as it guarantees convergence to the correct consensus (i.e. has zero probability of error) for arbitrary finite connected graphs. Besides, it only requires a limited amount of memory and communication by individual nodes (only four states). Some alternative decentralized algorithms require fewer states of memory or communication but fail to reach the correct consensus with strictly positive probability. For instance, the traditional voter model requires only two states of memory and communication. It is  however known that there are graphs for which the probability of error is a strictly positive constant, e.g. proportional to the number of nodes that initially held the minority state in the case of complete graphs (see \cite{peleg} for the general setting). Another example is the ternary protocol proposed in \cite{perron} for which it was shown that for complete graphs, the probability of error diminishes to zero exponentially with the number of nodes, but provides no improvement over the voter model for some other graphs (e.g. a path). 

In this paper, we provide an upper bound on the expected convergence time for solving the binary interval consensus on arbitrary connected graphs. This provides a unified approach for estimating the expected convergence time for particular graphs. The bound is tight in the sense that there exists a graph, namely the complete graph, for which the bound is achieved asymptotically for large number of nodes.
 
We demonstrate how the general upper bound can be instantiated for a range of particular graphs, including complete graphs, paths, cycles, star-shaped networks and Erd\"os-R\'enyi random graphs. Notice that the complete graph and the Erd\"os-R\'enyi random graph are good approximations of various unstructured and structured peer-to-peer networks and that star-shaped networks capture the scenarios where some node is a hub for other nodes. 

Our results provide insights into how the expected convergence time depends on \emph{the network structure} and \emph{the voting margin}, where the latter is defined as the difference between the fraction of nodes initially holding the majority state and the fraction of nodes initially holding the minority state. For the network structure, we found that the expected convergence time is determined by the spectral properties of some matrices that dictate the contact rates between nodes. For the voting margin, we found that there exist graphs for which the voting margin significantly affects the expected convergence time. 

\paragraph{Complete graph example} For concreteness, we describe how the voting margin affects the expected convergence time for the complete graph of $n$ nodes. Let us denote with $\alpha > 1/2$ the fraction of nodes that initially held the majority state, and thus $\alpha-(1-\alpha) = 2\alpha - 1$ is the voting margin. We found that that the convergence time $T$ satisfies
$$
\E(T)= \frac{1}{2\alpha-1}\log(n)(1+o(1)).
$$
Therefore, the expected convergence time is inversely proportional to the voting margin, and thus, goes to infinity as the voting margin goes to $0$. Hence, albeit the interval consensus guarantees convergence to the correct state, the expected convergence time can assume large values for small voting margins.

\paragraph{Outline of the Paper} In Section~\ref{sec:rel} we discuss the related work. Section~\ref{sec:pre} introduces the notation and the binary interval consensus algorithm considered in this paper. Section~\ref{sec:gen} presents our main result that consists of an upper bound on the expected convergence time that applies to arbitrary connected graphs (Theorem~\ref{thm:genbound}). Section~\ref{sec:app} instantiates the upper bound for particular graphs, namely complete graphs, paths, cycles, star-shaped networks and Erd\" os-R\' enyi random graphs, and compares with alternative analysis. We conclude in Section~\ref{sec:conc}. Some of the proofs are deferred to the appendix.

\section{Related Work}
\label{sec:rel}

In recent years there has been a large body of research on algorithms for decentralized computations over networks, under various constraints on the memory of individual nodes and communication between the nodes. For example, in the so called {\em quantized consensus} problem~\cite{srikant,Carli2010}, the goal is to approximately compute the mean of the values that reside at individual nodes, in a decentralized fashion, where nodes communicate quantized information. In \cite{nedic}, the authors provided bounds on the convergence time in the context of averaging algorithms where nodes exchange quantized information. 

The work that is most closely related to ours is \cite{benezit} where the authors showed that the so called interval consensus algorithm guarantees correctness for arbitrary finite connected graphs. In particular, their work shows that for solving the binary consensus problem, it suffices to use \emph{only two extra states} to guarantee convergence to the correct consensus in a finite time, for every finite connected graph. Our work advances this line of work by establishing the first tight characterizations of the expected convergence time for the binary interval consensus. 

Previous work on the binary consensus problem considered algorithms under more stringent assumptions on the number of states stored and communicated by individual nodes. The standard {\em voter model} is an algorithm where each node stores and communicates one of two states ($0$ or $1$), where each instigator node switches to the state observed from the contacted node. The voter model has been studied in the context of various graph topologies~\cite{donnelli,liggett,sood} and the probability of reaching the correct consensus (i.e. corresponding to the initial majority state) is known in closed-form for arbitrary connected graphs~\cite{peleg}. Specifically, the probability of reaching the correct consensus is proportional to the sum of degrees of the nodes that initially held the initial majority state. In particular, for the complete graphs, this means that the probability of reaching an incorrect consensus is proportional to the number of nodes that initially held the minority state. Moreover, for some network topologies, the convergence time of the voter model is known to be quadratic in the number of nodes, e.g. for a path~\cite{aldous}. In fact, It was shown in \cite{LB95} that it is impossible to solve the binary consensus problem without adding extra memory to encode the states of the nodes. Both~\cite{benezit} and \cite{MosselS10} show that adding one additional bit of memory is sufficient.

In \cite{perron}, the authors considered a ternary protocol for binary consensus problem where each node stores and communicates an extra state. It was shown that for the complete graph interactions, the probability of reaching the incorrect consensus is exponentially decreasing to $0$ as the number of nodes $n$ grows large, with a rate that depends on the voting margin. Moreover, if the algorithm converges to the right consensus, then the time it takes to complete is logarithmic in the number of nodes $n$, and is independent of the voting margin. Similar results have been derived in \cite{CG10} for the complete graph where, instead of using an extra state, each nodes polls more than one neighbour at a time and then chooses the opinion held by the majority of the polled neighbours. Notice that this is unlike to the binary interval consensus, for which we found that the expected convergence time, for the complete graph of $n$ nodes, is logarithmic in $n$, but with a factor that is dependent on the voting margin and going to infinity as the voting margin approaches zero. The main advantage of the binary interval consensus algorithm over both these protocols is the guaranteed convergence to the correct final state with probability $1$, albeit this seems to be at some expense with respect to the convergence time for some graphs.

Finally, we would like to mention that in a bigger picture, our work relates to the cascading phenomena that arise in the context of social networks \cite{kleinberg}; for example, in the viral marketing where an initial idea or behaviour held by a portion of the population, spreads through the network, yielding a wide adoption across the whole population~\cite{draief}.

\section{Algorithm and Notation}
\label{sec:pre}

In this section, we introduce the interval consensus algorithm for the binary consensus problem. Each node is assumed to be in one of the following four states $0$, $e_0$, $e_1$ and $1$, at every time instant. It is assumed that the states satisfy the following order relations $0 < e_0 < e_1 < 1$. Let us first describe a simple example to illustrate the updating protocol.
\paragraph{Example}  Assume that we have four nodes labelled $1,2,3,4$, forming a line network $1-2-3-4$, starting in state $(1,0,0,0)$. If the first interaction occurs between node $1$ and node $2$ then, as they disagree, both of them become undecided. More precisely, the state of node $1$ turns into $e_0$ indicating that she saw opinion $0$ and the state of node $2$ turns into $e_1$ indicating that she saw opinion $1$. The new vector of states becomes $(e_0,e_1,0,0)$. If the following interaction is between nodes $3$ and $4$ then nothing happens. If nodes $1$ and $2$ interact again then their states are swapped, i.e. the vector of states becomes $(e_1,e_0,0,0)$. Now suppose that nodes $2$ and $3$ interact then they swap their states, i.e. the vector of states becomes $(e_1,0,e_0,0)$. This transition indicates that node $2$ was undecided and saw opinion $0$ so she adopts $0$ whereas node $3$ saw an undecided node so she becomes undecided. If nodes $1$ and $2$ interact then the vector of states becomes $(e_0,0,e_0,0)$. This indicates that node $1$ saw an undecided node so becomes undecided whereas node $2$ was undecided so it adopted the opinion of node $1$. After this stage, the dynamics does not settle as if a node in $e_0$ and a node in $0$ interact then they swap their opinions but the number of nodes in each of the states $e_0$ and $0$ stays constant which is indicative of the fact that the initial majority held the opinion $0$.

We now describe the set of rules for updating the states of the nodes.
\paragraph{State update rules} The states held by the nodes are updated at pairwise contacts between nodes according to the following state update rules:
\begin{enumerate}
\item If a node in state $0$ and a node in state $1$ get in contact, they \emph{switch} their states to state $e_1$ and state $e_0$, respectively. 
\item If a node in state $e_0$ and a node in state $1$ get in contact, they  \emph{switch} their states to state $1$ and state $e_1$, respectively.
\item If a node in state $e_1$ and a node in state $0$ get in contact, they  \emph{switch} their states to state $0$ and state $e_0$, respectively.
\item If a node in state $e_0$ and a node in state $0$ get in contact, they \emph{swap} their states to state $0$ and $e_0$, respectively.
\item If a node in state $e_1$ and a node in state $1$ get in contact, they \emph{swap} their states to state $1$ and state $e_1$, respectively.
\item If a node in state $e_0$ and a node in state $e_1$ get in contact, they \emph{swap} their states to state $e_1$ and state $e_0$, respectively.
\end{enumerate}
For any other states of a pair of nodes that get in contact, their states remain unchanged. 
\paragraph{Temporal process of pairwise interactions} We admit the standard asynchronous communication model \cite{perron,boyd} where any pair of nodes $(i,j)$ interacts at instances of a Poisson process with rate $q_{i,j}\geq 0$. We denote with $V = \{1,2,\ldots,n\}$ the set of nodes. The interaction rates are specified by the matrix $Q=(q_{i,j})_{i,j\in V}$ assumed to be symmetric\footnote{We assume that $i$ contacts $j$ at a rate $p_{i,j}\geq 0$ so that $i$ and $j$ interact at a rate $q_{ij}=p_{ij}+p_{ji}$. Therefore, defining $P=(p_{ij})_{i,j\in V}$, we have $Q=P+P^T$ is symmetric matrix, i.e. $q_{i,j}=q_{j,i}$ for every $i,j\in V$ and $q_{i,i} = 0$ for every $i\in V$.}. The transition matrix $Q$ induces an undirected graph $G=(V,E)$ where there is an edge $(i,j)\in E$ if and only if $q_{i,j} > 0$. We assume that graph $G$ is connected.
\paragraph{Two convergence phases} The state update rules ensure that in a finite time, a final state is reached in which all nodes are either in state $0$ or state $e_0$ (state $0$ is initial majority). We distinguish two phases in the convergence to the final state, which will be a key step for our analysis of the expected convergence time that relies on separately analyzing the two phases. The two convergence phases are defined as follows: 
\begin{enumerate}
\item[Phase~1] (depletion of state $1$). This phase begins at the start of the execution of the algorithm and lasts until none of the nodes is in state $1$. Whenever a node in state $0$ and a node in state $1$ get in contact, they switch to states $e_1$ and $e_0$, respectively. It is therefore clear that the number of nodes holding the minority state (state $1$) decreases to $0$ in a finite amount of time and from that time onwards, the number of nodes in state $0$ remains equal to the difference of the (initial) number of nodes in state $0$ and state $1$. 
\item[Phase~2]  (depletion of state $e_1$). This phase follows the end of phase~1 and lasts until none of the nodes is in state $e_1$. In this phase, the number of nodes in state $e_1$ decreases following each contact between a node in state $e_1$ and a node in state $0$. Since, in this phase, no interaction between a pair of nodes results in increasing the number of nodes in state $e_1$, there are eventually no nodes in state $e_1$.
\end{enumerate} 
The duration of each of the two phases is ensured to be finite for arbitrary finite connected graphs by the definition of the state update rules where swapping of the states enables that state $1$ nodes get in contact with state $0$ nodes and similarly, enables that state $e_1$ nodes get in contact with state $0$ nodes.

%
 
\paragraph{Additional notation} We denote by $S_i(t)$ the set of nodes in state $i\in \{0,e_0,e_1,1\}$ at time $t$. With a slight abuse of notation, in some cases we will use the compact notation $|S_i|\equiv |S_i(0)|$, $i=0,1$, which should be clear from the context. We define $\alpha\in(1/2,1]$ as the fraction of nodes that initially hold state $0$, assumed to be the initial majority. Therefore, $|S_0|=\alpha n$ and $|S_1|=(1-\alpha)n$. 

\section{General Bound for the Expected Convergence Time}
\label{sec:gen}
In this section we present our main result that consists of an upper bound on the expected convergence time for arbitrary connected graphs. \\
The bound is in terms of eigenvalues of a set of matrices $Q_S$ that is defined using the transition matrix $Q$ as follows. Let $S$ be a non-empty subset of the set of vertices $V$ of size smaller than $n$ and let $S^c = V\setminus S$. We consider the matrix $Q_S=(q^S_{i,j})_{i,j\in V}$ that is derived from the contact rate matrix $Q$ as follows
\begin{equation}
q^S_{i,j} = \left\{
\begin{array}{ll}
-\sum_{l\in V} q_{i,l}, & i=j\\
q_{i,j}, & i\in S^c, j\neq i\\
0, & i\in S, j\neq i.
\end{array}
\right .
\label{equ:QS}
\end{equation} 
We first establish that eigenvalues of the matrices $Q_S$, for $S\subset V$ non-empty, are strictly negative. This will be a key property that ensures finiteness of our bound which we present later in this section.\\

\begin{lemma} For every finite graph $G$, we define $\delta(Q,\alpha)$ as follows 
\begin{eqnarray}\nonumber
\delta(Q,\alpha)&=&\min_{S\subset V, \frac{|S|}{n}\in[2\alpha-1, \alpha]} |\lambda_{Q_S}|\\ \label{Cauchy interlacement}
&=&\min_{S\subset V, |S|=(2\alpha-1)n} |\lambda_{Q_S}|\:.
\end{eqnarray}
where $\lambda_{Q_S}$ is the largest eigenvalue of $Q_S$. Then, for all $S$ non-empty, $\lambda_{Q_S}<0$ and thence $\delta(Q,\alpha)>0$.
\label{lem:spec}
\end{lemma}

Note that identity (\ref{Cauchy interlacement}) is a direct consequence of the Cauchy interlacing theorem for principal submatrices of orthogonal matrices \cite[Theorem 4.3.8, p. 185]{Horn90}.

We next present our main result that establishes an upper bound on the expected convergence time that holds for arbitrary connected graphs. Before stating the result, notice that at the end of phase~1 none of the nodes are in state $1$, $(2\alpha-1)n$ nodes are in state $0$, and the remaining $2(1-\alpha)n$ nodes are in either state $e_0$ or state $e_1$. At the end of phase~2, there are exactly $(2\alpha-1)n$ nodes in state $0$ and $2(1-\alpha) n$ nodes in state $e_0$. The following theorem establishes a general bound for the expected duration of each convergence phases in terms of the number of nodes $n$ and the parameter $\delta(Q,\alpha)$, which we introduced in Lemma~\ref{lem:spec}.\\

\begin{theorem} Let $T_1$ be the first instant at which all the nodes in state $1$ are depleted. Then, 
$$
\E(T_1) \leq \frac{1}{\delta(Q,\alpha)}(\log n + 1).
$$ 
Furthermore, letting $T_2$ be the time for all the nodes in state $e_1$ to be depleted, starting from an initial state with no nodes in state $1$, we have
$$
\E(T_2) \leq \frac{1}{\delta(Q,\alpha)}(\log n + 1).
$$ 
In particular, if $T$ is the  first instant at which none of the nodes is in either state $e_1$ or state $1$, then
$$
\E(T)\leq \frac{2}{\delta(Q,\alpha)}(\log n + 1).
$$
\label{thm:genbound}
\end{theorem}

It is worth noting that the above theorem holds for every positive integer $n$ and not just asymptotically in $n$.

The proof of the theorem is presented in Section~\ref{sec:pro} and here we outline the main ideas. The proof proceeds by first separately considering the two convergence phases. For phase~1, we characterize the evolution over time of the probability that a node is in state $1$, for every node $i\in V$. This amounts to a ``piecewise'' linear dynamical system. Similarly, for phase~2, we characterize the evolution over time of the probability that a node is in state $e_1$, for every given node $i\in V$, and show that this also amounts to a  ``piecewise'' linear dynamical system. The proof is then completed by using a spectral bound on the expected number of nodes in state $1$, for phase $1$, and in state $e_1$, for phase~2, which is then used to establish the asserted results.
\paragraph{Tightness of the bounds} The bound for the convergence time of the first phase asserted in Theorem~\ref{thm:genbound} is tight in the sense that there exist graphs for which the asymptotically dominant terms of the expected convergence time and the corresponding bound are either equal (complete graph) or equal up to a constant factor (star network). The bound for the second phase is not tight as we prove an upper bound for it using the worst case for an initial configuration for the start of phase $2$. \\
In what follows we provide proofs for Lemma~\ref{lem:spec} and Theorem~\ref{thm:genbound}.
\subsection{Proof of Lemma~\ref{lem:spec}}
Let $S$ be a non-empty subset of $V$. First, for the trivial case $S=V$ the matrix $Q_S$ is diagonal with diagonal elements $(-\sum_{j\in V} q_{ij})_{i\in V}$ which are all negative since the graph is connected (each node has at least one neighbour). Now let $S$ such that $|S|<n$. Note that every eigenvalue $\lambda$ and the associated eigenvector $\vec{x}$ of the matrix $Q_S$ satisfy the following equations
\begin{equation}
\begin{array}{ll}
\lambda x_i = -q_i x_i, & \hbox{ for } i\in S\\
\lambda x_i = - q_i x_i + \sum_{l\in V} q_{i,l}x_l, & \hbox{ for } i \in S^c
\end{array}
\label{equ:lambda}
\end{equation}
where $q_i := \sum_{l\in V} q_{i,l}$, for every $i\in V$.\\
On the one hand, it is clear from the form of the matrices $Q_S$, given by (\ref{equ:QS}), that for every $i\in S$, $\lambda = -q_i$ is an eigenvalue of $Q_S$. Since by assumption, the transition matrix $Q$ induces a connected graph $G$, we have that for every $i\in V$ there exists a $j\in V$ such that $q_{i,j} > 0$. Hence, it follows that $\lambda < 0$. \\
On the other hand, if $\lambda \neq -q_i$ for any $i\in S$, it is clear from (\ref{equ:lambda}) that $x_l = 0$ for every $l\in S$. In fact, since $Q$ is symmetric, the remaining eigenvalues of $Q_S$ are the eigenvalues of the symmetric matrix $M_S=(m^S_{i,j})_{i,j\in S^c}$ defined by
$$
m^S_{i,j} = \left\{
\begin{array}{ll}
-\sum_{l\in V} q_{i,l}, & i=j\in S^c\\
q_{i,j}, & i,j\in S^c, j\neq i
\end{array}
\right .
$$
Let $\lambda$ be such an eigenvalue of $Q_S$ and let $\vec{x}$ be the corresponding eigenvector and, without loss of generality, assume that $||\vec{x}||_2^2=\sum_{i\in V} x_i^2=\sum_{i\in S^c} x_i^2=1$.  Note that $\lambda=\lambda\vec{x}^T\vec{x}=\vec{x}^TQ\vec{x}$. Since $Q$ is symmetric, we have
\begin{eqnarray}
-\lambda&=&\sum_{i\in S^c,j\in V} q_{i,j} x_i^2 +\sum_{i,j\in S^c}q_{i,j} x_ix_j\nonumber\\
&=&\sum_{i\in S^c, j\in S} q_{i,j} x_i^2 -\sum_{i,j\in S^c}q_{i,j} x_i(x_i-x_j)\nonumber\\
&=&\sum_{i\in S^c,j\in S} q_{i,j} x_i^2 -\frac{1}{2}\sum_{i,j\in S^c}q_{i,j} (x_i-x_j)^2.\label{equ:quad}
\end{eqnarray}
Therefore, it is clear that $\lambda\leq 0$ with $\lambda= 0$ only if
$$
\sum_{i\in S^c,j\in S} q_{i,j} x_i^2 +\frac{1}{2}\sum_{i,j\in S^c}q_{i,j} (x_i-x_j)^2=0.
$$
Let $W\subset S^c$ be such that $x_i\neq 0$, for $i\in W$, and $x_i = 0$, for $i\in S^c \setminus W$. Since $\vec{x}$ is an eigenvector, then $W$ is non empty. If $\lambda=0$, then 
$$
\sum_{i\in W,j\in S} q_{i,j} x_i^2 +\sum_{i\in W,j\in S^c\setminus W}q_{i,j} x_i^2+\frac{1}{2}\sum_{i,j\in W} q_{i,j}(x_i-x_j)^2=0.
$$
The above implies that there are no edges between $S$ and $W$, and that there are no edges between $W$ and $S^c\setminus W$, i.e. $W$ is an isolated component, which is a contradiction since $Q$ corresponds to a connected graph. Therefore, $\lambda<0$, which proves the lemma.
\subsection{Proof of Theorem~\ref{thm:genbound}}\label{sec:pro}
We first separately consider the two convergence phases and then complete with a step that applies to both phases.
\paragraph{Phase~1: Depletion of nodes in state $1$} We describe the dynamics of the first phase through the following indicators of node states. Let $Z_i(t)$ and $A_i(t)$ be the indicators that node $i$ is in state $0$ and $1$ at time $t$, respectively. The indicator of being in either state $e_0$ or state $e_1$ at time $t$ is encoded by $A_i(t)=Z_i(t)=0$. The system state evolves according to a continuous-time Markov process $(Z(t),A(t))_{t\geq 0}$, where $A(t)=(A_i(t))_{i\in V}$ and $Z(t)=(Z_i(t))_{i\in V}$, with the transition rates given as follows
$$
(Z,A) \rightarrow \left\{
\begin{array}{lll}
(Z-\epsilon_i,A-\epsilon_j) & \mbox{with rate} &   q_{i,j} Z_i A_j\\
(Z-\epsilon_i+\epsilon_j,A) & \mbox{with rate} &   q_{i,j} Z_i (1-A_j-Z_j)\\
(Z,A-\epsilon_i+\epsilon_j) & \mbox{with rate} &   q_{i,j} A_i (1-A_j-Z_j)
\end{array}
\right .
$$
where $i,j\in V$ and $\epsilon_i$ is the $n$-dimensional vector whose elements are all equal to $0$ but the $i$-th element that is equal to $1$.\\
Since $Q$ is a symmetric matrix, we have for every $i\in V$ and $t\geq 0$,
\begin{eqnarray*}
\frac{d}{dt}\E(A_i(t))&=&-\sum_{j\in V} q_{i,j} \E(A_i(t) Z_j(t)) 
 - \sum_{j\in V} q_{i,j}\E\left(A_i(t)(1-A_j(t)-Z_j(t))\right)\\
&&+\sum_{j\in V} q_{i,j}\E\left(A_j(t)(1-A_i(t)-Z_i(t))\right)\\
\end{eqnarray*}
or, equivalently,
\begin{eqnarray*}
\frac{d}{dt}\E(A_i(t)) &=&-\left(\sum_{l\in V} q_{i,l}\right)\E(A_i(t)) + \sum_{j\in V} q_{i,j}\E\left(A_j(t)(1-Z_i(t))\right).
\end{eqnarray*}
Let us now consider the behaviour of the set $S_0(t)$ of nodes in state $0$, i.e. $S_0(t)=\{i\in V: Z_i(t)=1\}$. From the above dynamics, we see that there are intervals $[t_k,t_{k+1})$ during which the set $S_0(t)$ does not change (the instants $t_k$ are stopping times of the Markov chain describing the evolution of the algorithm). Let $S_k\subset V$ be the set of nodes in state $0$, for $t\in [t_k,t_{k+1})$, and let $S_k^c = V\setminus S_k$, i.e. $S_0(t)=S_k$ and $V\setminus S_0(t) = S_k^c$, for $t\in[t_k,t_{k+1})$. We then can write, for $t\in[t_k,t_{k+1})$,
\begin{eqnarray}
\frac{d}{dt}\E_k(A_i(t))= -\left(\sum_{l\in V} q_{i,l}\right)\E_k(A_i(t))+\left\{
\begin{array}{ll}
\sum_{j\in V} q_{i,j}\E_k\left(A_j(t)\right), & i \in S_k^c  \\ 
0,& i\in S_k
\end{array}
\right.
\end{eqnarray}
where $\E_k$ is the expectation conditional on the event $\{S_0(t)=S_k\}$. In a matrix form, this gives
$$
\frac{d}{dt}\E_k(A(t))=Q_{S_k}\E_k(A(t)), \hbox{ for } t_k \leq t < t_{k+1},
$$
where $Q_{S_k}$ is given by (\ref{equ:QS}).\\
Solving the above differential equation, we have
$$\E_k(A(t))=e^{Q_{S_k}(t-t_k)} \E_k(A(t_k)), \hbox{ for } t_k \leq t < t_{k+1}.
$$
Using the strong Markov property, it is not difficult to see that
$$
\E(A(t))=\E\left[e^{\lambda(t)}A(0)\right], \hbox{ for } t\geq 0,
$$
where
$$
\lambda(t) = Q_{S_k}(t-t_k)+\sum_{l=0}^{k-1}Q_{S_l}(t_{l+1}-t_l), \hbox{ for } t_k \leq t < t_{k+1}.
$$
Note that $\lambda(t) $ is a random matrix that depends on the stopping times $t_k$.
\paragraph{Phase~2: Depletion of nodes in state $e_1$} To describe the dynamics in the second phase, let $B_i(t)$ be the indicator that a node $i\in V$ is in state $e_1$ at time $t$. The notation $Z_i(t)$ has the same meaning as in phase~1, thus $Z_i(t)$ is the indicator that node $i\in V$ is in state $0$ at time $t$. The indicator that a node $i\in V$ is in state $e_0$ at time $t$ is encoded by $B_i(t)=Z_i(t)=0$. \\
The dynamics in this phase reduces to a continuous-time Markov process $(Z(t),B(t))_{t\geq 0}$, where $Z(t)=(Z_i(t))_{i\in V}$ and $B(t)=(B_i(t))_{i\in V}$, with the transition rates given as follows, for $i,j\in V$,
$$
(Z,B) \rightarrow \left\{
\begin{array}{lll}
(Z-\epsilon_i+\epsilon_j,B-\epsilon_j) & \mbox{ with rate } &   q_{i,j} Z_i B_j\\
(Z-\epsilon_i+\epsilon_j,B) & \mbox{ with rate }&   q_{i,j} Z_i (1-B_j-Z_j)\\
(Z,B-\epsilon_i+\epsilon_j) & \mbox{ with rate }&   q_{i,j} B_i (1-B_j-Z_j).
\end{array}
\right .
$$
From this, we have for every $i\in V$ and $t\geq 0$,
\begin{eqnarray*}
\frac{d}{dt}\E(B_i(t))&=&-\sum_{i\in V} q_{i,j} \E(B_i(t) Z_j(t)) - \sum_{j\in V} q_{i,j}\E\left(B_i(t)(1-Z_j(t)-B_j(t))\right)\\
&&+\sum_{j\in V} q_{i,j}\E\left(B_j(t)(1-Z_i(t)-B_i(t))\right).
\end{eqnarray*}
Therefore, for every $i\in V$ and $t\geq 0$,
\begin{eqnarray*}
\frac{d}{dt}\E(B_i(t))
&=&-\left(\sum_{l\in V} q_{i,l}\right)\E(B_i(t)) + \sum_{j\in V} q_{i,j}\E\left(B_j(t)(1-Z_i(t))\right).
\end{eqnarray*}
Similar to the first phase, we see that there are intervals $[t'_k,t'_{k+1})$ during which the set $S_0(t)$ does not change (the instants $t'_k$ are stopping times). Let $S'_k$ be such that $S_0(t)=S'_k$, for $t\in[t'_k,t'_{k+1})$. Similarly to the first phase, we have
$$
\E(B(t)) = \E\left[e^{\lambda'(t)} B(t'_0)\right], \hbox{ for } t\geq 0,
$$
where $\lambda'(t)$ is a random matrix given by
$$
\lambda'(t) = Q_{S'_k}(t-t'_k)+\sum_{l=0}^{k-1}Q_{S'_l}(t'_{l+1}-t'_l), \hbox{ for } t_k'\leq t < t'_{k+1}.
$$
Note that $t'_0=T_1$ is the instant at which phase 2 starts (phase 1 ends).
\paragraph{Duration of a phase} In both phases, the process of interest is of the form 
$$
\E(Y(t)) = \E\left[e^{\lambda(t)} Y(0)\right], \hbox{ for } t\geq 0,
$$
where for a (random) positive integer $m > 0$ and a sequence $0=t_0\leq t_1 \leq \cdots \leq t_m$, we have  
$$
\lambda(t) = Q_{S_k}(t-t_k)+\sum_{l=0}^{k-1}Q_{S_l}(t_{l+1}-t_l), \hbox{ for } t_k\leq t<t_{k+1},\ k = 0,1,\ldots,m-1.
$$
For phase~1, $Y(t)\equiv A(t)$  while for phase~2, $Y(t)\equiv B(t)$. Using techniques as in \cite[Chapter 8]{DM10}, we have for or every $t\geq 0$,
\begin{eqnarray*}
 ||\E(Y(t))||_2
\leq  \E\left[\left|\left|e^{\lambda(t)}Y(0)\right|\right|_2 \right]
&\leq & \E\left[\left|\left|e^{\lambda(t)}\right|\right|\: ||Y(0)||_2\right]\\
&\leq &\E\left[||e^{Q_{S_k}(t-t_k)}|| \prod_{l=0}^{k-1} ||e^{Q_{S_l}(t_{l+1}-t_l)}|| \:||Y(0)||_2\right]\\
&\leq & e^{-\delta(Q,\alpha) t}\E\left(||Y(0)||_2\right)\leq \sqrt{n}\: e^{-\delta(Q,\alpha) t}\:
\end{eqnarray*}
where $||\cdot||$ denotes the matrix norm associated to the Euclidean norm $||\cdot||_2$. In the above, we used Jensen's inequality in the first inequality, followed by the property of matrix norms for the second and third inequalities, then Lemma~\ref{lem:spec} and finally the fact that $Y$ is a $n$-dimensional vector with elements taking values in $\{0,1\}$. \\
Furthermore, combining with Cauchy-Schwartz's inequality, we have
\begin{eqnarray*}
\sum_{i\in V} \E(Y_i(t))\leq ||\E(Y(t))||_2\: ||{\bf 1}||_2
\leq n\:e^{-\delta(Q,\alpha) t}, \hbox{ for every } t \geq 0.
\end{eqnarray*}
where  ${\bf 1}=(1,\dots,1)^T$. Therefore, we have
\begin{eqnarray*}
\P(Y(t)\neq{\bf 0}) \leq \sum_{i\in V} \E(Y_i(t))
\leq  n\:e^{-\delta(Q,\alpha) t},\ \hbox{ for every } t \geq 0.
\end{eqnarray*}
Let $T_0$ be the time at which $Y(t)$ hits ${\bf 0}=(0,\dots,0)^T$, which corresponds to $T_1$ for the process $A(t)$ and $T_2$ for the process $B(t)$. Then, we have 
\begin{eqnarray*}
\E(T_0)=\int_0^{\infty}\P(T_0>t) dt
&=& \int_0^{\infty}\P(Y(t)\neq {\bf 0}) dt\\
 &\leq & \frac{\log(n)}{\delta(Q,\alpha)}+n\:\int_{\frac{\log(n)}{\delta(Q,\alpha)}}^{\infty}e^{-\delta(Q,\alpha)  t}dt\\
 &=&\frac{\log(n)+1}{\delta(Q,\alpha)}\:
\end{eqnarray*}
which completes the proof of the theorem.
\section{Application to Particular Graphs}
\label{sec:app}
In this section we instantiate the bound of Theorem~\ref{thm:genbound} for particular networks including complete graphs, paths, cycles, star-shaped networks and Erd\" os-R\' enyi random graphs. For all these cases, we compare with alternative computations and find that our bound is of exactly the same order as the expected convergence time with respect to the number of nodes. For the complete graph, we also examine the expected convergence time as the voting margin goes to zero.

\subsection{Complete Graph}
\label{sse:complete}
We consider the complete graph of $n > 1$ nodes where each edge $e\in E$ is activated at instances of a Poisson process with rate $1/(n-1)$, i.e. we have $q_{i,j}=1/(n-1)$ for all $i,j\in V$ such that $i\neq j$.\\

\begin{lemma} For the complete graph of $n > 1$ nodes and every fixed $\alpha \in (1/2,1]$, we have
$$
\delta(Q,\alpha) \geq (2\alpha-1).
$$
\label{lem:comp1}
\end{lemma}
\begin{proof} For the complete graph, the matrix $Q_S$ is as follows
$$
q^S_{i,j} = \left\{
\begin{array}{rl}
- 1, & i=j\\
\frac{1}{n-1}, & i\in S^c, j \neq i\\
0, & i\in S, j \neq i.
\end{array}
\right .
$$
First of all $-1$ is an eigenvalue of order $|S|$. In addition it is not difficult to see that the vector $\vec{x}$ such that $x_i=0$ for $i\in S$ and $x_i=1$ for $i\in S^c$ is an eigenvector of matrix $Q_S$ with the eigenvalue $-\frac{|S|}{n-1}$. Since in each of the two convergence phases, the matrices $Q_{S_k}$ are such that $|S_k|\geq (2\alpha-1) n$, we have $\frac{|S_k|}{n-1}\geq (2\alpha-1)\frac{n}{n-1}$. Finally note that the remaining eigenvalues are the eigenvalues of the matrix $M_S=(m^S_{i,j})_{i,j\in S^c}$ defined by
$$
m^S_{i,j} = \left\{
\begin{array}{ll}
-1, & i=j\in S^c\\
\frac{1}{n-1}, & i,j\in S^c, j\neq i
\end{array}
\right .
$$
One can rewrite $M_S$ as $M_S=-\frac{n}{n-1} I+\frac{1}{n-1}J\:,$
where $I$ is the identity matrix and $J$ is the matrix with all its entries equal to $1$. Therefore the remaining eigenvalue of $Q_S$ is $-\frac{n}{n-1}$ with mutiplicity $n-|S|-1$ and $\delta(Q,\alpha)\geq(2\alpha-1)$.
\end{proof}

Combining the last lemma with Theorem~\ref{thm:genbound}, we have the following corollary.\\

\begin{corollary} For the complete graph of $n > 1$ nodes, the expected duration of phase $i=1$ and $2$ satisfies
$$
\E(T_i) \leq \frac{1}{2\alpha-1}(\log(n) + 1).
$$
\label{cor:comp}
\end{corollary}
In the following, we will show that the latter bound is asymptotically tight, for large number of nodes $n$, for convergence phase~1.
\paragraph{Comparison with an alternative analysis} For complete graphs, the convergence time can be studied by an analysis of the underlying stochastic system that we describe in the following. \\
We first consider the convergence phase~1. Let $0= \tau_0 \leq \tau_1 \leq \cdots \leq \tau_{|S_1|}$ denote the time instances at which a node in state $0$ and a node in state $1$ get in contact. Recall that $|S_0|$ and $|S_1|$ denote the initial number of nodes in state $0$ and state $1$, respectively. It is readily observed that $|S_0(t)| = |S_0| - i$ and $|S_1(t)|=|S_1|-i$, for $\tau_{i}\leq t < \tau_{i+1}$ and $1\leq i < |S_1|$. \\
It is not difficult to observe that $\tau_{i+1}-\tau_i$, $i = 0, 1, \ldots,|S_1|-1$, is a sequence of independent random variables such that for each $0\leq i < |S_1|$, $\tau_{i+1}-\tau_i$ is a minimum of a sequence of $(|S_0|-i)(|S_1|-i)$ i.i.d. random variables with exponential distribution with mean $n-1$. Therefore, the distribution of $\tau_{i+1}-\tau_i$ is exponential with mean $1/\mu_i$, for $0\leq i < |S_1|$, where $\mu_i= (|S_0|-i)(|S_1|-i)/(n-1)$. In particular, we have $\E(T_1)=\sum_{i=0}^{|S_1|-1}\mu_i^{-1}$, i.e.
\begin{equation}
\E(T_1) = (n-1)\sum_{i=0}^{|S_1|-1} \frac{1}{(|S_0|-i)(|S_1|-i)}.
\label{equ:ep1}
\end{equation}

\begin{figure*}[t]
\vspace*{-4cm}
\begin{center}
\hspace*{-0.8cm}\psfig{figure=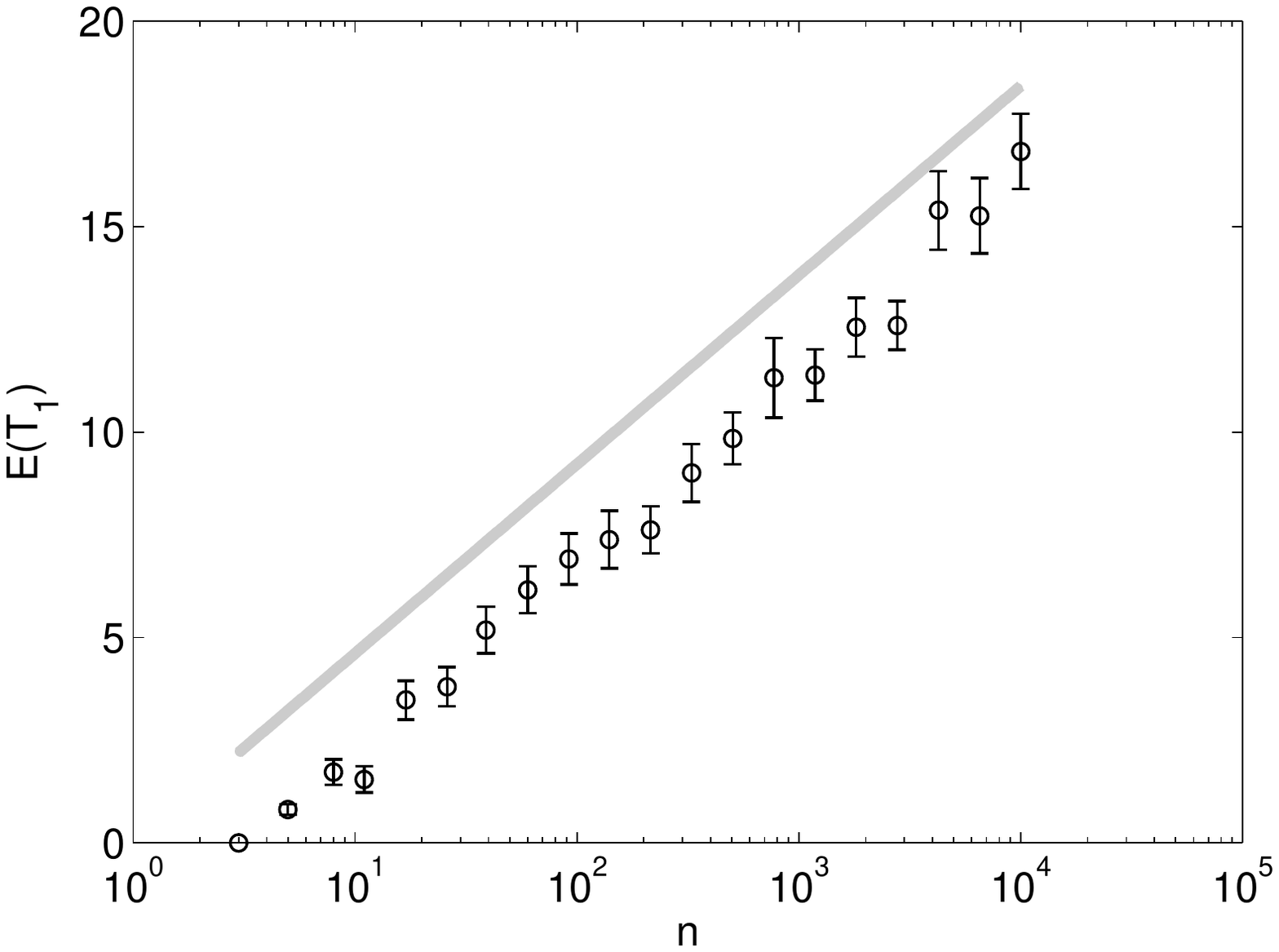,height=5.5in,width=3in}\hspace*{-1.0cm}
\hspace*{0cm}\psfig{figure=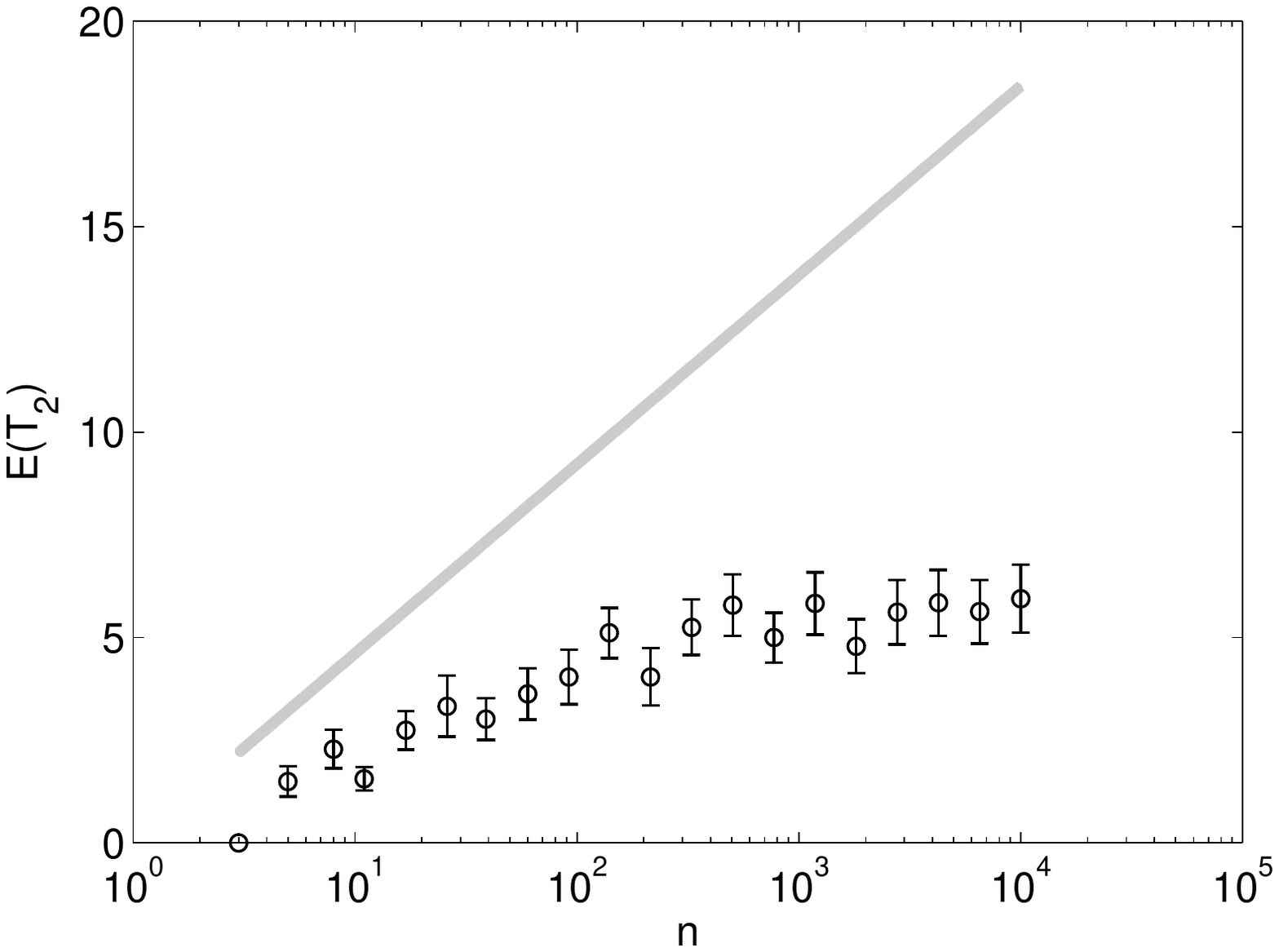,height=5.5in,width=3in}
\vspace{-4.5cm}
\caption{Complete graph: the expected duration of convergence phases vs. the number of nodes $n$. The initial majority state is held by $\lceil \alpha n\rceil$ nodes, where $\alpha = 3/4$. The solid line is the asymptote $\log(n)/(2\alpha - 1)$ while the bars are simulation results with $95\%$ confidence intervals.}
\label{fig:compN}
\end{center} 
\end{figure*}

The above considerations result in the following proposition.\\

\begin{proposition} For every complete graph with $n > 1$ nodes and initial state such that $|S_0|>|S_1| > 0$, the duration of the first convergence phase is a random variable $T_1$ with the following expected value
\begin{equation}
\E(T_1) = \frac{n-1}{|S_0|-|S_1|}\left(H_{|S_1|}+H_{|S_0|-|S_1|} - H_{|S_0|}\right)
\label{equ:ep11}
\end{equation}
where $H_k=\sum_{i=1}^k \frac{1}{i}$. Furthermore, for every fixed $\alpha \in (1/2,1]$, we have
$$
\E(T_1) = \frac{1}{2\alpha - 1}\log(n) + O(1).
$$
\label{lem:comp2}
\end{proposition}

From the result of the proposition, we observe that the expected duration of the first phase is $\log(n)/\delta(Q,\alpha)$, asymptotically for large $n$, where $\delta(Q,\alpha)$ is given in Lemma~\ref{lem:comp1}, thus matching the upper bound of Theorem~\ref{thm:genbound}.

For the prevailing case of a complete graph, we can characterize the expected convergence time of the first phase as $\alpha$ approaches $1/2$, i.e. as the voting margin $2\alpha - 1$ approaches $0$ from above. We first consider the limit case where initially there is an equal number of nodes in state $0$ and state $1$, i.e. $|S_0| = |S_1|$. From (\ref{equ:ep1}), it is straightforward to note
$$
\E(T_1) = \frac{\pi^2}{6}n (1+o(1)).
$$
Therefore, we observe that in case of an initial draw, i.e. equal number of state $0$ and state $1$ nodes, the expected duration of the first phase scales linearly with the network size $n$. Note that in this case, the second phase starts with nodes in state $e_0$ and state $e_1$ and obviously no majority can follow.

 We now discuss the case where $|S_0|-|S_1|$ is strictly positive but small. To this end, let $\mu_n$ denote the voting margin, i.e. $\mu_n = (|S_0| - |S_1|)/n$. From (\ref{equ:ep11}), is easy to observe that
$$
\E(T_1)=\frac{1}{\mu_n}\log(n \mu_n) + O(1).
$$ 
Therefore, we note that for the voting margin $\mu_n = O(1/n)$, $\E(T_1) = \Theta(n)$ while for $\mu_n > 0$ a fixed constant, we have $\E(T_1) = \Theta(\log(n))$. For the intermediate values of the voting margin, say for $\mu_n = 1/n^a$, for $0 < a < 1$, we have $\E(T_1) = \frac{1-a}{2}n^a \log(n)$.

Finally, we compare the bound $\log(n)/(2\alpha-1)$ with simulation results, in Figure~\ref{fig:compN}. We observe that the bound is tight for phase~1 and not tight for phase~2, due to our choice of the initial condition in phase 2. Note also that Figure~\ref{fig:compN} indicates that the expected duration of convergence phase~2 scales as $\Theta\left(\log(\log n)\right)$.

\subsection{Paths}

We consider a path of $n > 1$ nodes where each edge is activated at instances of a Poisson process of rate $1$. Therefore, the contact rate matrix $Q$ is given by $q_{i,i+1}=1$, for $i=1,\dots,n-1$, $q_{i,i-1}=1$, for $i=2,3,\ldots,n$, and all other elements equal to $0$.\\

\begin{lemma} For a path of $n > 1$ nodes, we have, for $\alpha\in(1/2,1]$,
\begin{eqnarray*}
\delta(Q,\alpha) &=& 2\left(1-\cos\left(\frac{\pi}{4(1-\alpha)n+1}\right)\right)\\
&= & \frac{\pi^2}{16(1-\alpha)^2 n^2}(1+o(1)).\\
\end{eqnarray*}
\label{thm:path}
\end{lemma}

\begin{figure*}[t]
\vspace*{-4cm}
\begin{center}
\hspace*{-1.0cm}\psfig{figure=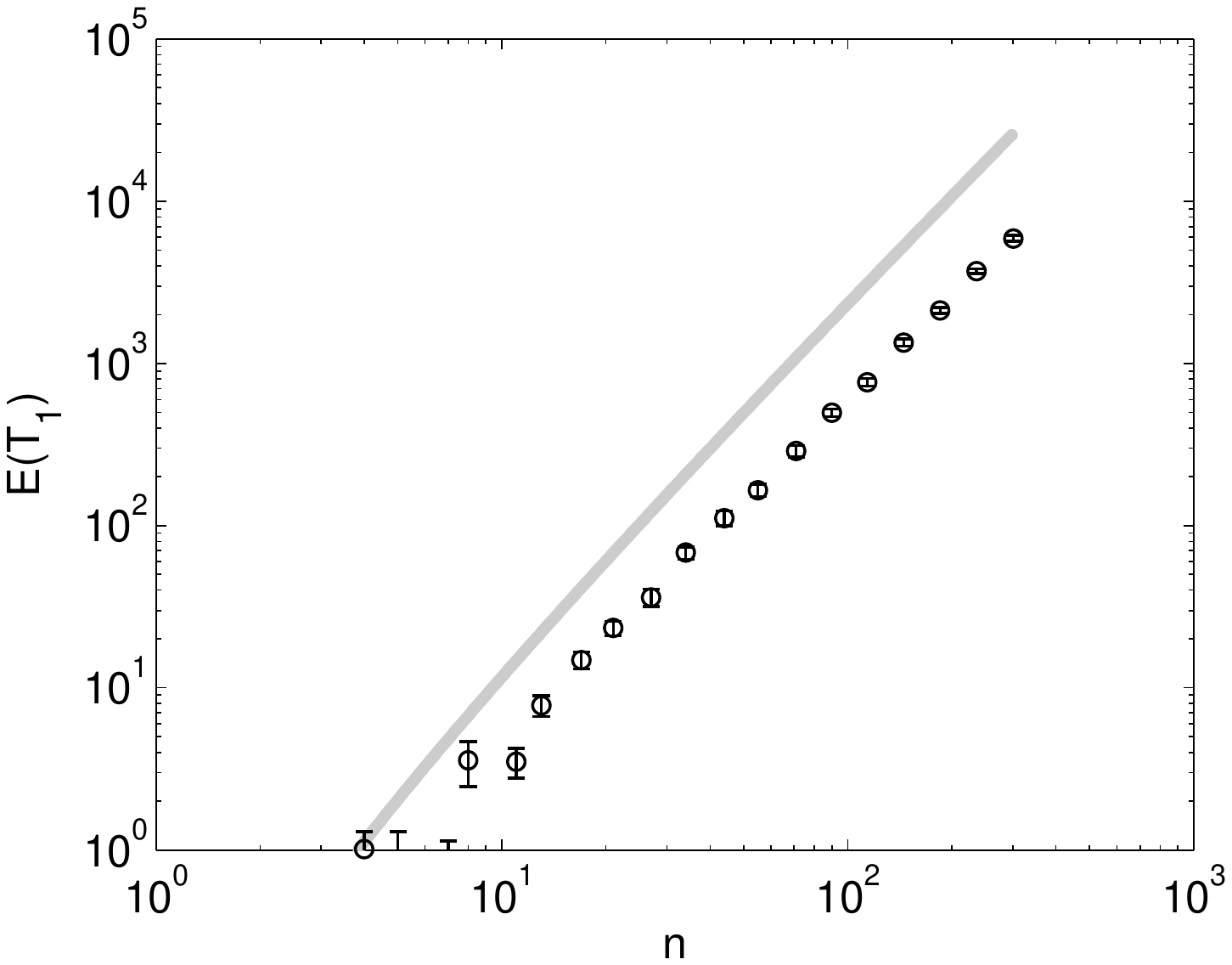,height=5.5in,width=3.3in}\hspace*{-1.9cm}
\hspace*{0cm}\psfig{figure=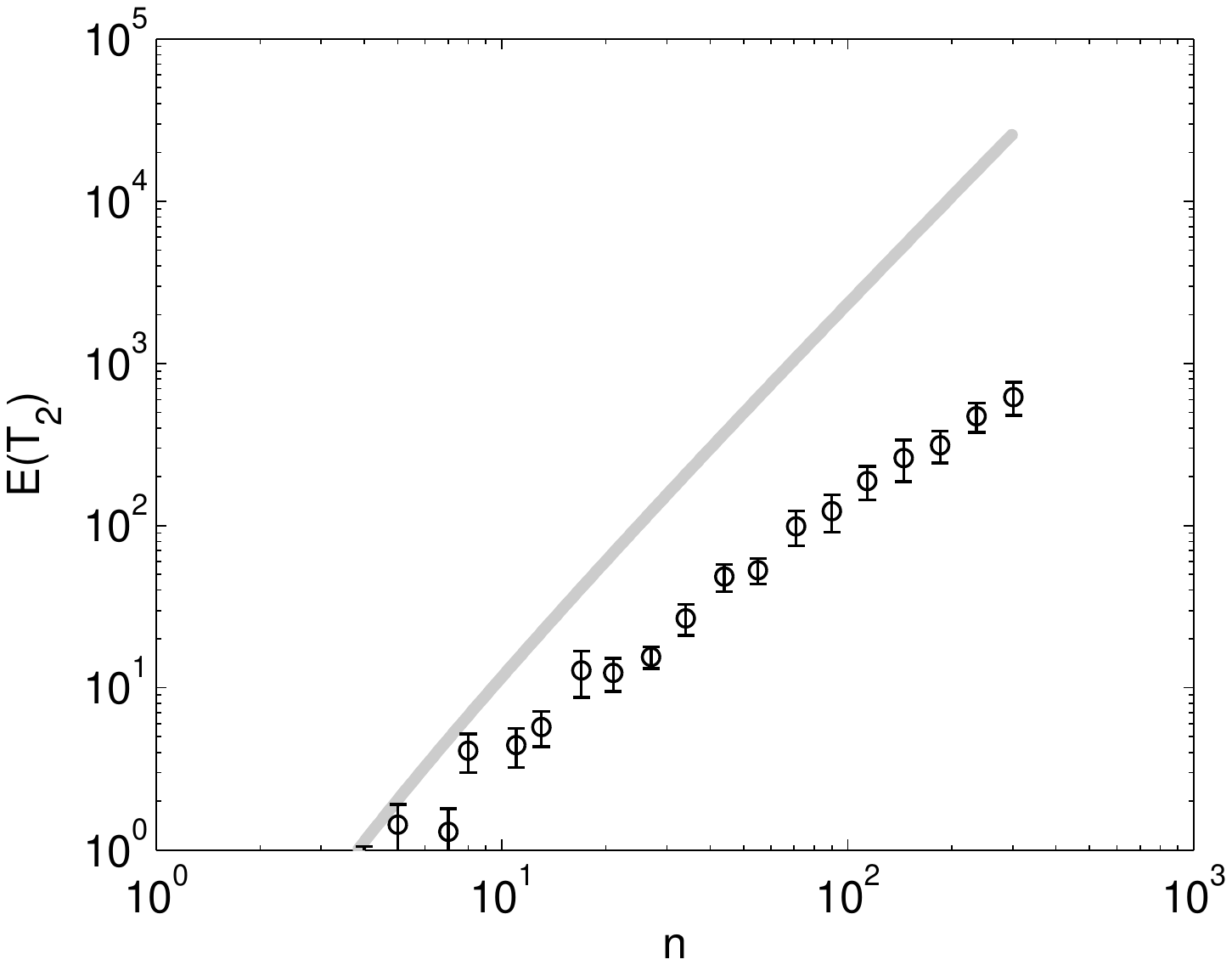,height=5.5in,width=3.3in}
\vspace{-4cm}
\caption{Path: the expected duration of convergence phases vs. the number of nodes $n$, for $\alpha = 3/4$. The solid line is the asymptote in Corollary~\ref{cor:path} while the bars are simulation results with $95\%$ confidence intervals.}
\label{fig:path}
\end{center} 
\end{figure*}

The proof is provided in Appendix~\ref{sec:delta-path}.The previous lemma, together with Theorem~\ref{thm:genbound}, yields the following result.\\

\begin{corollary} For a path of $n > 1$ nodes and $\alpha \in (1/2,1)$, we have for phase $i=1$ and $2$,
$$
\E(T_i) \leq \frac{16(1-\alpha)^2}{\pi^2}n^2 \log(n) + O(1).
$$
\label{cor:path}
\end{corollary}

Finally, we compare the asymptotic bound with simulation results in Figure~\ref{fig:path}. The results indicate that the bound is  rather tight for phase~1 and is not tight for phase~2.

\subsection{Cycles}

We consider a cycle of $n > 1$ nodes where each edge is activated at instances of a Poisson process with rate $1$. Therefore, the contact rate matrix $Q$ is given by $q_{i,i+1}=1$, for $i=1,\dots,n-1$, $q_{i,i-1}=1$, for $i=2,3,\ldots,n$, $q_{1,n}=q_{n,1}=1$, and all other elements equal to $0$.\\

\begin{lemma} For a cycle network of $n > 1$ nodes, we have, for $\alpha\in(1/2,1]$,
\begin{eqnarray*}
\delta(Q,\alpha) &=& 2\left(1 - \cos\left(\frac{\pi}{2(1-\alpha)n+1}\right)\right)\\
&= &\frac{\pi^2}{4(1-\alpha)^2 n^2}(1+o(1)).\\
\end{eqnarray*}
\label{thm:cycle}
\end{lemma}

The proof is provided in Appendix~\ref{sec:delta-cycle}.\\

\begin{figure*}[t]
\begin{center}
\vspace*{-4cm}
\hspace*{-1cm}\psfig{figure=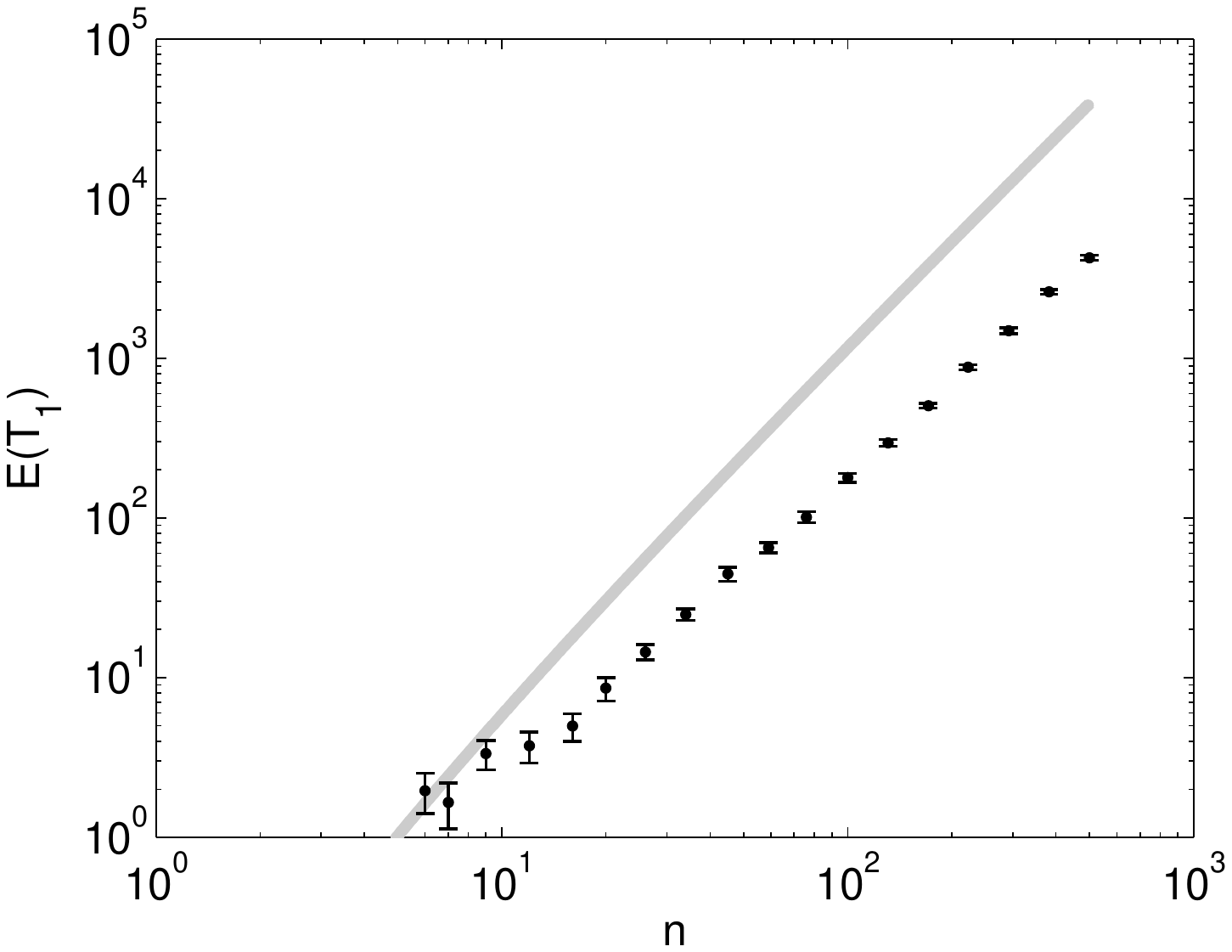,height=5.5in,width=3.3in}\hspace*{-1.9cm}
\hspace*{-0cm}\psfig{figure=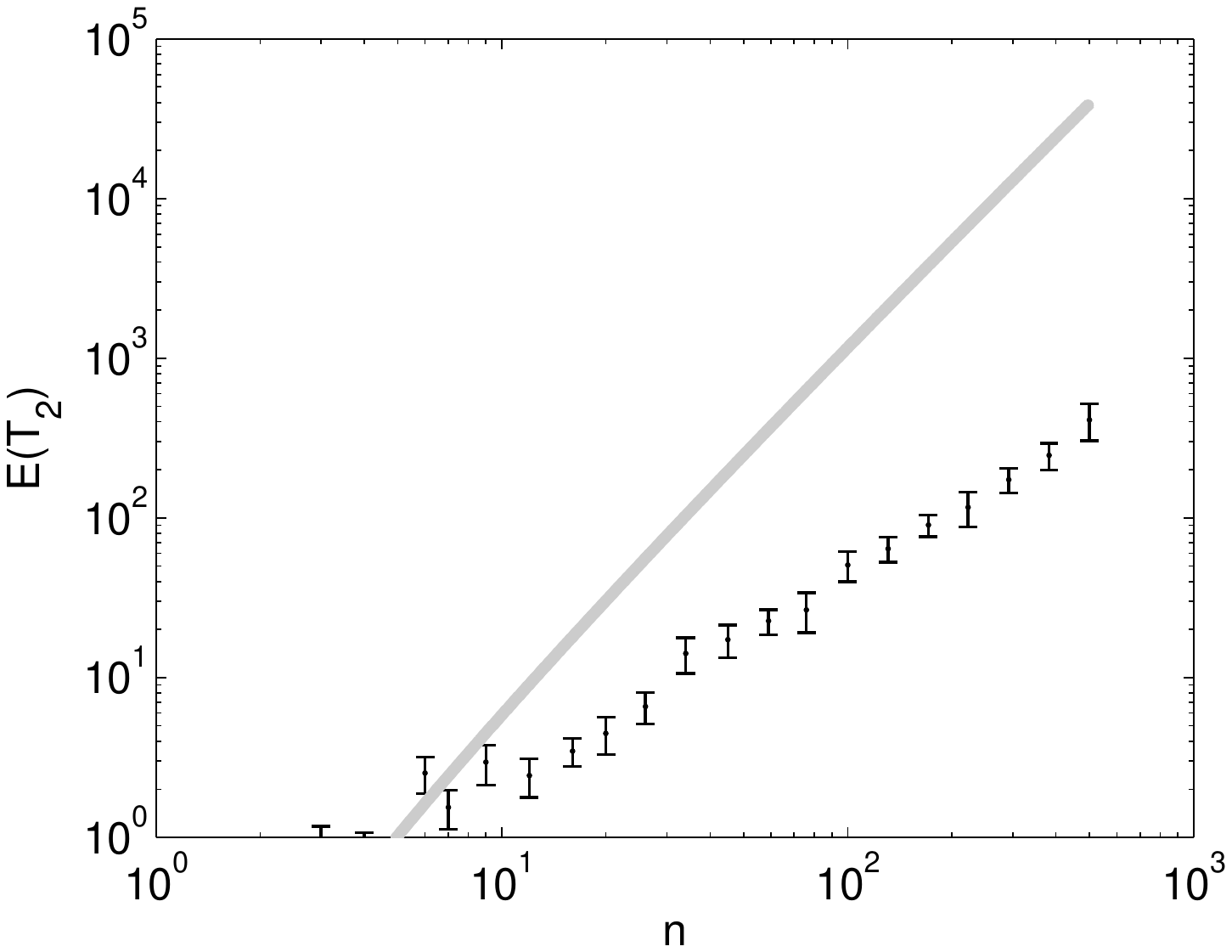,height=5.5in,width=3.3in}
\vspace{-4.5cm}
\caption{Cycle: the expected durations of convergence phases vs. the number of the nodes $n$. The initial state is such that state $0$ is held by a set of $\lceil \alpha n\rceil $ consecutive nodes along the cycle with $\alpha = 3/4$.}
\label{fig:cycleN}
\end{center} 
\end{figure*}

\begin{corollary} For the cycle with $\alpha \in (1/2,1)$, we have for phase $i=1$ and $2$,
$$
\E(T_i) \leq \frac{4(1-\alpha)^2}{\pi^2}n^2 \log(n) + O(1).
$$
\end{corollary}

Finally, we compare the last bound with simulation results in Figure~\ref{fig:cycleN}. Similar as in other cases, we observe that the bound has the same scaling with the number of nodes as the expected duration of convergence phase~1, and is not tight for convergence phase~2. 

\subsection{Star-Shaped Networks}

We consider a star-shaped network that consists of a hub node and $n-1$ leaf nodes. Without loss of generality, let the hub node be node $1$ and let $i = 2,3,\ldots, n$ be the leaf nodes. The contacts between a leaf node and the hub are assumed to occur at instances of a Poisson process of rate $1/(n-1)$. This setting is motivated in practice by networks where a designated node assumes the role of an information aggregator to which other nodes are connected and this aggregator node has access capacity of rate $1$. The elements of matrix $Q$ are given by $q_{1,i} = q_{i,1}=1/(n-1)$, for $i=2,3,\ldots,n$ and other elements equal to $0$.

We have the following lemma for the star-shaped network that we defined above.\\

\begin{lemma} For the star network of $n > 1$ nodes, we have
\begin{eqnarray*}
\delta(Q,\alpha) &=& \frac{n}{2(n-1)}\left(1-\sqrt{1-\frac{4(2\alpha-1)}{n}}\right)\\
&\geq & \frac{2\alpha-1}{n}
\end{eqnarray*}
where the inequality is tight for large $n$.\\
\label{thm:star}
\end{lemma}

The proof is provided in Appendix~\ref{sec:delta-star}. The previous lemma yields the following corollary.\\
\begin{corollary} For the star network with $n > 1$ nodes and every fixed $\alpha \in (1/2,1]$, the expected duration of phase $i=1$ and $2$ satisfies
$$
\E(T_i) \leq \frac{1}{2\alpha - 1}n(\log(n)+1).
$$
\label{cor:star}
\end{corollary}
\begin{figure*}[t]
\begin{center}
\vspace*{-4cm}
\hspace*{-1cm}\psfig{figure=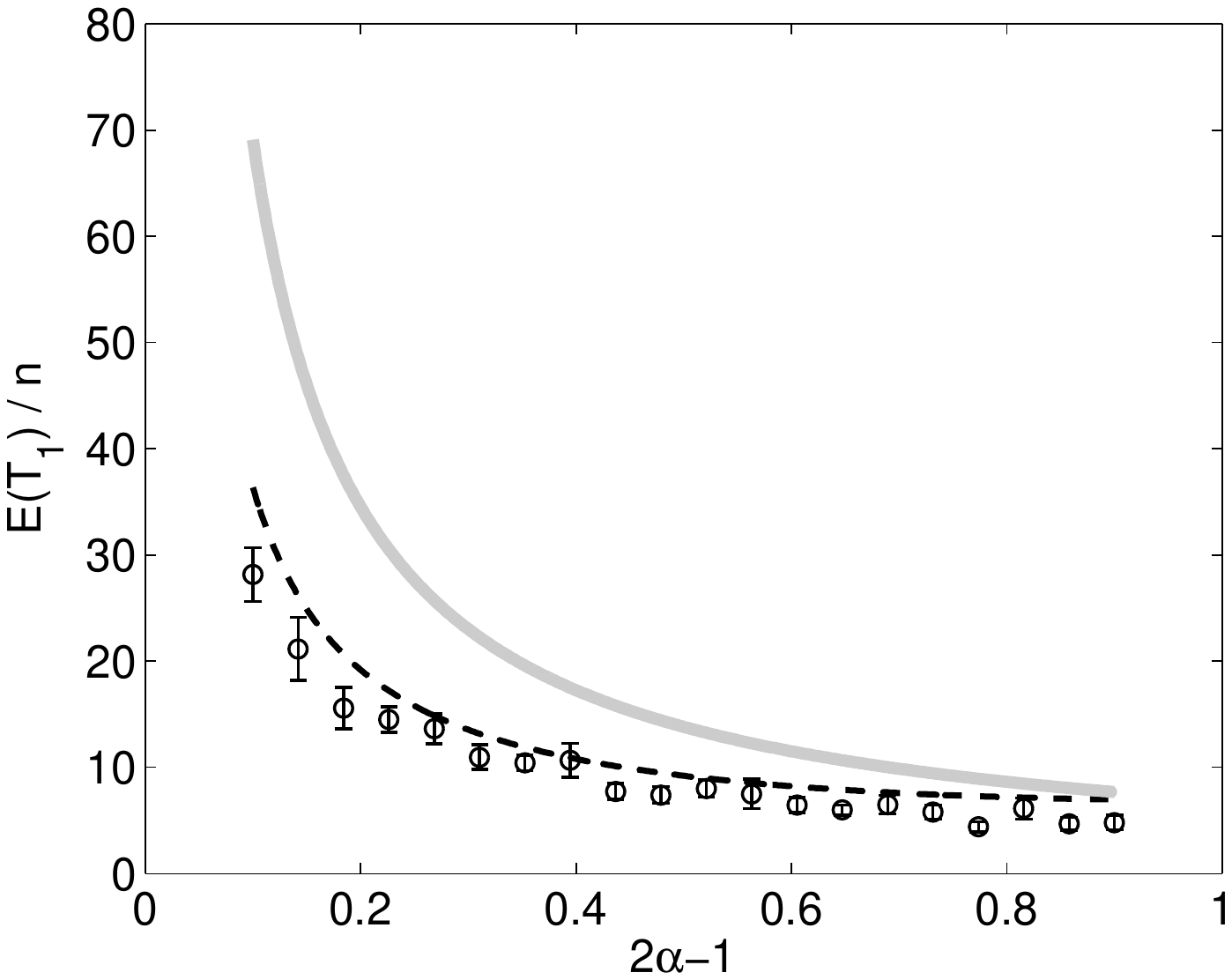,height=5.5in,width=3.3in}\hspace*{-1.9cm}
\hspace*{0cm}\psfig{figure=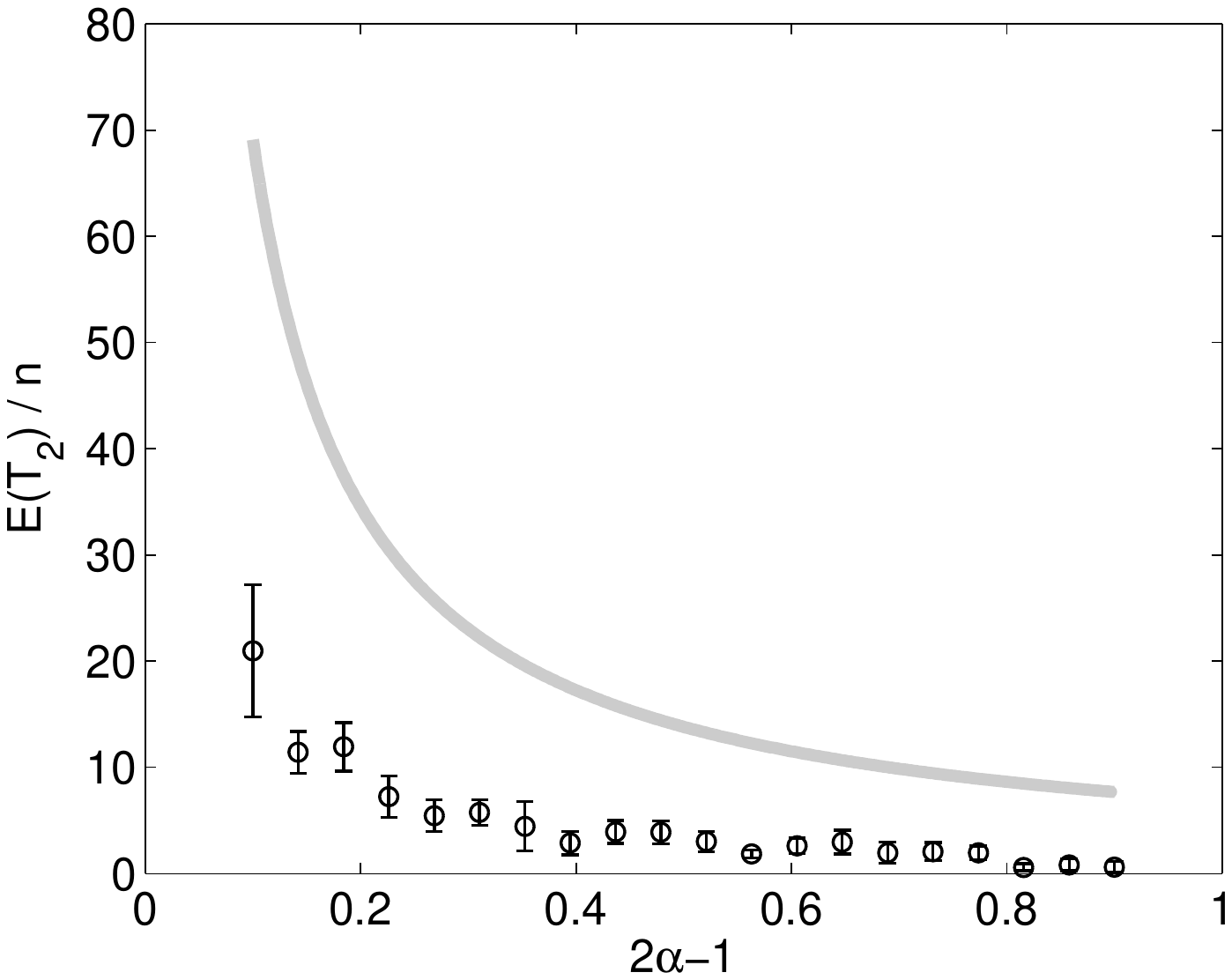,height=5.5in,width=3.3in}
\vspace*{-4.5cm}
\caption{Star-shaped network: expected duration of convergence phases versus the voting margin $2\alpha-1$, for $n=1000$. The solid curves indicate $\log(n)/(2\alpha-1)$; the dashed line indicates $\log(n)/[(2\alpha-1)(3-2\alpha)]$; the bars indicate $95\%$-confidence intervals of estimates obtained by simulations.}
\label{fig:starN}
\end{center} 
\end{figure*}
\paragraph{Comparison with an alternative analysis for phase~1} For the star-shaped network of $n$ nodes, we can compute the exact asymptotically dominant term of the expected duration of phase~1, for large $n$, which is presented in the following proposition.
\begin{proposition} For the star-shaped network of $n$ nodes, the expected time to deplete nodes in state $1$ satisfies
\begin{equation}
\E(T_1) = \frac{1}{(2\alpha-1)(3-2\alpha)}n\log(n) + O(n).
\label{equ:starlimit}
\end{equation}
\label{pro:star-phase1}
\end{proposition}
The proof is provided in Appendix~\ref{sec:star-phase1}.\\
Notice that the dominant term in Proposition~\ref{pro:star-phase1} is smaller than the upper bound in Corollary~\ref{cor:star} for the factor $1/(3-2\alpha)$.
\paragraph{Remark} We only consider the expected convergence time for phase~1. Similar analysis could be pursued for phase~2 but is more complicated, because the lumping of the states as done in the proof for phase~1 cannot be made.\\
Finally, we compare our bound with simulation results in Figure~\ref{fig:starN}. The results indicate that the bound of Corollary~\ref{cor:star} is not tight. We also observe that the asymptote in Proposition~\ref{pro:star-phase1} conforms well with simulation results.
\subsection{Erd\" os-R\' enyi Random Graphs}\label{sec:er}
We consider random graphs for which the matrix of contact rates $Q$ is defined as follows. Given a parameter $p_n \in (0,1)$ that corresponds to the probability that a pair of nodes interact with a strictly positive rate, we define the contact rate of a pair of nodes $i,j\in V$, $i\neq j$, as follows
$$
q_{i,j} = \frac{1}{(n-1)p_n}X_{i,j}
$$
where $X_{i,j}$ is a sequence of i.i.d. random variables such that $\P(X_{i,j}=1) = 1-\P(X_{i,j} = 0) = p_n$, for every $i,j\in V$, $j\neq i$. The rates are normalized with the factor $1/(n-1)p_n$, so that for each node, the interaction rate with other nodes is $1$.

Furthermore, we assume that $p_n$ is chosen such that, for a constant $c > 1$,
$$
p_n = c\frac{\log(n)}{n}
$$
which ensures that the induced random graph is connected with high probability. \\
We have the following lemma.
\begin{lemma} Suppose $c > \frac{2}{2\alpha-1}$ and $\alpha \in (1/2,1]$. We then have, with high probability that
\begin{equation}
\delta(Q,\alpha)\geq (2\alpha-1)\varphi^{-1}\left(\frac{2}{c(2\alpha-1)}\right) + O\left(\frac{1}{\log n}\right),
\label{equ:invd}
\end{equation}
where $\varphi^{-1}(\cdot)$ is the inverse function of $\varphi(x) = x\log(x) + 1 - x$, for $x \in [0,1]$.
\label{lem:er}
\end{lemma}
\begin{figure*}[t]
\begin{center}
\vspace*{-3.5cm}
\hspace*{-1cm}\psfig{figure=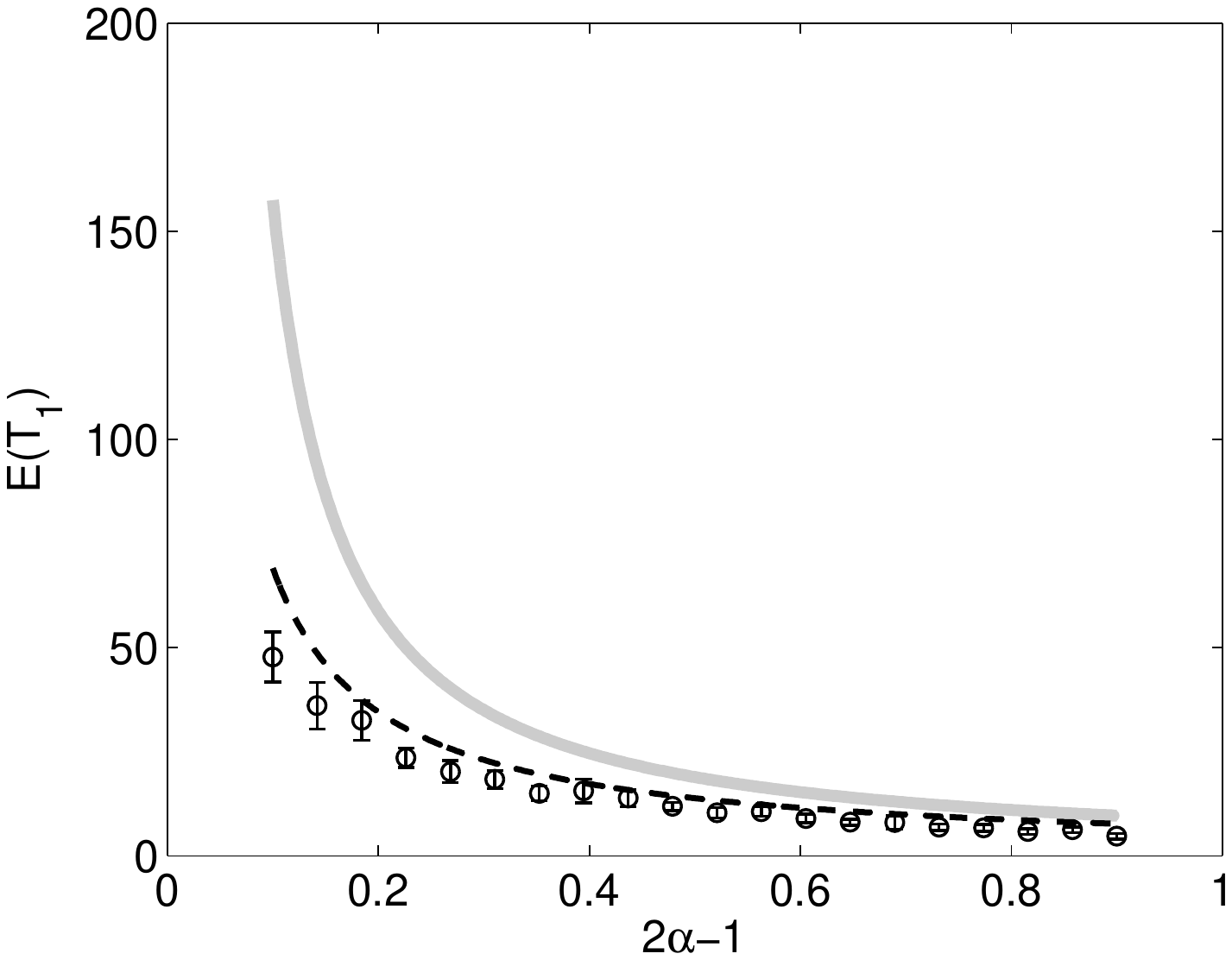,height=5.5in,width=3.3in}\hspace*{-1.9cm}
\hspace*{0cm}\psfig{figure=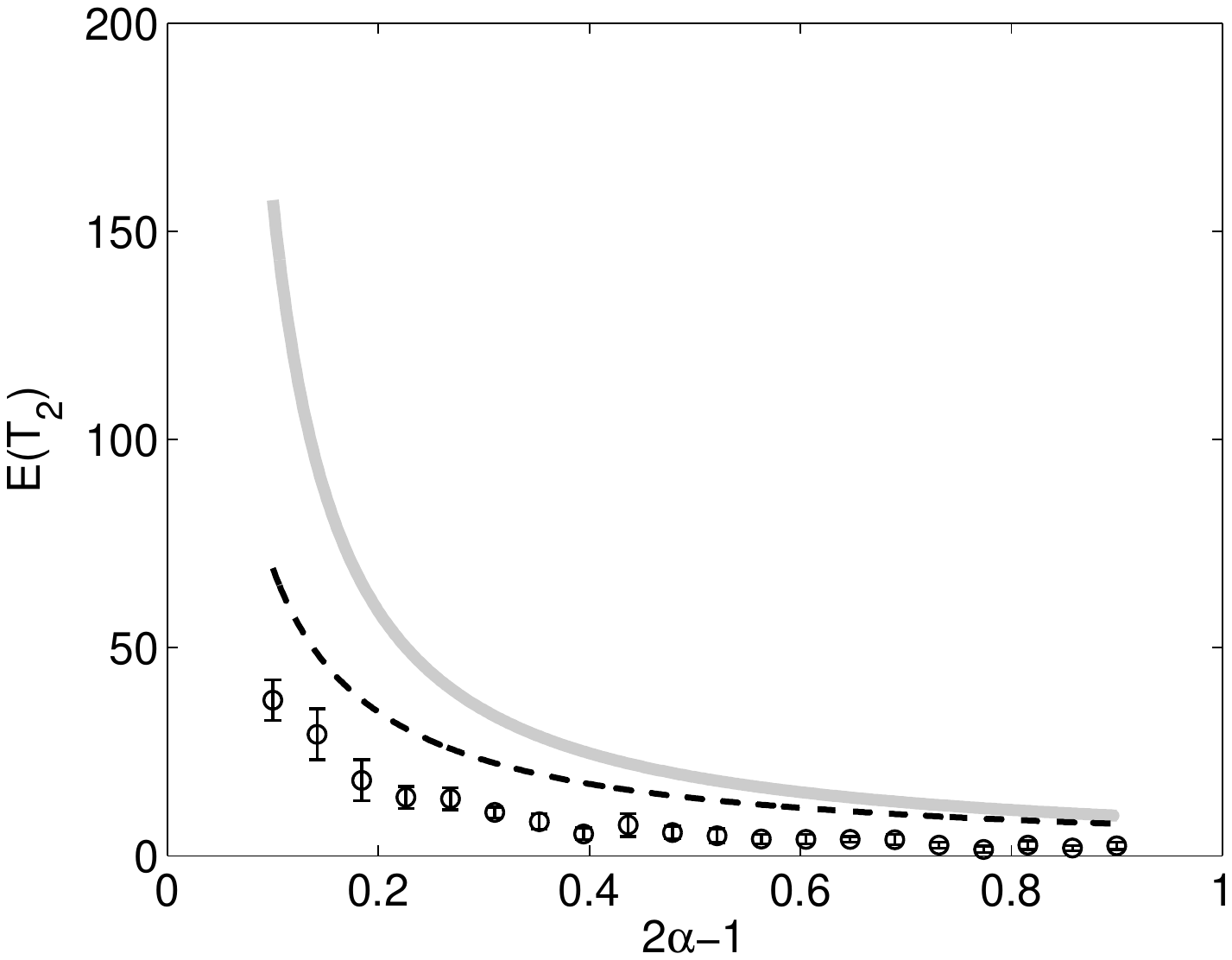,height=5.5in,width=3.3in}
\vspace*{-4cm}
\caption{Erd\" os-R\' enyi random graphs: the expected duration of convergence phases vs. the voting margin $2\alpha-1$, for $n=1000$ and $c=100$. The solid curves indicate the bound of Corollary~\ref{equet2}; the dashed lines indicate $\log(n)/(2\alpha-1)$; the bars indicate $95\%$-confidence estimates.}
\label{fig:erN}
\end{center} 
\end{figure*}
The proof is provided in Appendix~\ref{pro:er}.\\
From the last lemma and Theorem~\ref{thm:genbound}, we have the following corollary.\begin{corollary} Under $c > \frac{2}{2\alpha - 1}$ and $\alpha \in (1/2,1]$, we have for the duration of phase $i=1$ and $2$, 
\begin{equation}
T_i \leq \frac{1}{(2\alpha-1)\varphi^{-1}\left(\frac{2}{c(2\alpha-1)}\right)}\log(n) + O(1)
\label{equet2}
\end{equation}
with high probability.\\
\label{cor:er}
\end{corollary}

\paragraph{Remark} We note the following intuitive observation: the asserted bound for the expected convergence time for each of the phases boils down to that of the complete graph, for large expected degree of a node, i.e. large $c$. Indeed, this holds because for every fixed $\alpha \in (1/2,1]$, the term $\varphi^{-1}(\frac{2}{c(2\alpha-1)})$ goes to $1$ as $c$ grows large.\\

Finally, we compare the bound of Corollary~\ref{cor:er} with estimates obtained by simulations in Fig.~\ref{fig:erN}. The results confirm that the bound is indeed a bound and that it is not tight, which is because of a bounding technique that we used in the proof. 
  
\section{Conclusion}
\label{sec:conc}

We established an upper bound on the expected convergence time of the binary interval consensus that applies to arbitrary connected graphs. We showed that for a range of particular graphs, the bound is of exactly the same order as the expected convergence time with respect to the network size. The bound provides insights into how the network topology and the voting margin affect the expected convergence time. In particular, we showed that there exist network graphs for which the expected convergence time becomes much larger when the voting margin approaches zero. The established bound provides a unifying approach to bound the expected convergence time of binary interval consensus on arbitrary finite and connected graphs.

An important direction of future work is to consider lower bounds on the convergence time. In particular, it would be of interest to better understand how to fine tune the interaction parameters $q_{ij}$ to achieve the best possible convergence time for a given connected graph  under given memory and communication constraints.

\section*{Ackowledgement} MD is supported by QNRF through grant NPRP 09-1150-2-448. MD holds a Leverhulme Trust Research Fellowship RF/9/RFG/2010/02/08.

\appendix
\section{Characterization of $\delta$ for Particular Graphs}\label{sec:delta}

\subsection{Proof of Lemma~\ref{thm:path} (Path)}\label{sec:delta-path} Recall that the matrix $Q$ is the tridiagonal matrix with $q_{i,i+1}=1$, for $i=1,2,\ldots,n-1$, $q_{i,i-1}=1$, for $i=2,3,\ldots,n$, and all other elements equal to $0$.

We will separately consider four cases depending on whether the respective end-node $1$ and $n$ is in $S$ or $S^c$. In the following, we denote with $\xi_A(\lambda)$, the characteristic polynomial of a matrix $A$.\\

\noindent \emph{Case 1}: $1\in S$ and $n\in S^c$. In this case, we repeatedly expand the matrix $\lambda I - Q_S$ along the rows $i\in S$, using Laplace's formula, to obtain 
\begin{equation}
\xi_{Q_{S}}(\lambda)= (\lambda+1)(\lambda + 2)^{|S|-1}\xi_B(\lambda)
\label{equ:p1}
\end{equation}
where the matrix $B$ is a block-diagonal matrix with blocks $B_1$, $B_2$, \ldots, $B_b$, for $1 < b \leq |S|$ that are symmetric tridiagonal matrices of the form 
\begin{equation}
\left( \begin{array}{lllllll}
-2 + c_1 &  1 &  0  & 0 & \cdots & 0 & 0\\
1  & -2  & 1  & 0 & \cdots & 0 & 0\\
0  &  1  & -2 & 1 & \cdots & 0 & 0\\
\cdots & \cdots & \cdots & \cdots & \cdots & \cdots & \cdots\\
0 & 0 & 0 & 0 & \cdots & 1 & -2 + c_2
\end{array} \right)
\label{equ:Btridiagonal}
\end{equation}
where $c_1 = c_2 = 0$, for $i = 1,2,\ldots,b-1$ (type 1), and $c_1 = 0$ and $c_2 = 1$, for $i = b$ (type 2).

Since $B$ is a block-diagonal matrix, notice that $\xi_B(\lambda) = \prod_{i=1}^b \xi_{B_i}(\lambda)$. Hence, together with (\ref{equ:p1}), we have that $Q_S$ the largest eigenvalue of a block that is either of type $1$ or type $2$.  

The eigenvalues of tridiagonal matrices of the form (\ref{equ:Btridiagonal}) are well known for some values of the parameters $c_1$ and $c_2$, see e.g. \cite{Y05}. In particular, for a $m\times m$ tridiagonal matrix of type~1, i.e. for $c_1 = c_2 = 0$, we have eigenvalues
\begin{equation}
\lambda_k = -2 \left(1-\cos\left(\frac{\pi k}{m+1}\right)\right), \ k = 1,2,\ldots,m.
\label{equ:ltype1}
\end{equation}
For a $m\times m$ tridiagonal matrix of type~2, i.e. for $c_1 = 0$ and $c_2 = 1$, we have eigenvalues
$$
\kappa_k = -2\left(1-\cos\left(\frac{(2k-1)\pi}{2m+1}\right)\right),\ k = 1,2,\ldots,m.
$$

It is readily checked that the largest eigenvalue is $\kappa_1$ with $m = |S^c|$. This corresponds to the case where the nodes in the set $S$ are $1,2,\ldots,|S|$ (i.e. form a cluster). \\

\noindent \emph{Case~2}: $1\in S^c$ and $n\in S$. In this case, we have that (\ref{equ:p1}) holds but the blocks of matrix $B$ redefined so that $c_1 = 1$ and $c_2 = 0$, for $B_1$, and $c_1 = c_2 = 0$, for $B_2$, $B_3$, $\ldots$, and $B_b$. It is readily observed that the eigenvalues of the block matrix $B_1$ are the same as for $c_1 = 0$ and $c_2 = 1$ (type~1 tridiagonal matrix in Case~1). Hence, the largest eigenvalue is same as under Case~1. Notice that in this case, it corresponds to taking nodes $n-|S|+1$, $n-|S|+2$, $\ldots$, and $n$ to be in the set $S$.\\

\noindent \emph{Case~3}: $1\in S^c$ and $n\in S^c$. In this case, by the same arguments as in Case~1, we have
$$
\xi_{Q_S}(\lambda)= (\lambda + 2)^{|S|}\xi_B(\lambda)
\label{equ:p2}
$$
where $B$ is a block-diagonal matrix with blocks $B_1$, $B_2$, \ldots, $B_b$, $1 < b \leq |S^c|$, which are of the form (\ref{equ:Btridiagonal}) such that $c_1 = 1$ and $c_2 = 0$ for $B_1$, $c_1 = c_2 = 0$, for $B_2$, $B_3$, $\ldots$, and $B_b$, and $c_1 = 0$ and $c_2 = 1$, for $B_b$. In this case, the largest eigenvalue is $\kappa_1$ with $m = |S^c| - 1$.\\

\noindent \emph{Case~4}: $1\in S$ and $n\in S$. In this case, we have
$$
\xi_{Q_S}(\lambda)= (\lambda+1)^2(\lambda + 2)^{|S|-2}\xi_B(\lambda)
\label{equ:p3}
$$
where $B$ is a block-diagonal matrix with blocks $B_1$, $B_2$, \ldots, $B_b$, $1 < b \leq |S^c|$ that are all of the form (\ref{equ:Btridiagonal}) with $c_1 = c_2 = 0$. In this case, the largest eigenvalue is $\lambda_1$ with $m = |S^c|$.\\
Finally, we observe that for each of the four cases, since $|S^c| \leq 2(1-\alpha)n$, we can take $\delta(Q,\alpha)$ as asserted in the lemma, which completes the proof. 

\subsection{Proof of Lemma~\ref{thm:cycle} (Cycle)}\label{sec:delta-cycle} The proof is similar to that for a path in Section~\ref{sec:delta-path}. We will see that for the cycle, we will deal with blocks of tridiagonal matrices of the form (\ref{equ:Btridiagonal}) with $c_1 = c_2 = 0$. Recall that for the cycle, we have matrix $Q$ such that $q_{i,i+1}=1$, for $i = 1,2,\ldots,n-1$, $q_{i-1,i}=1$, for $i=2,3,\ldots,n$, $q_{1,n}=q_{n,1} = 1$, and all other elements equal to $0$.\\
Again, we separately consider the following four cases.

\noindent \emph{Case~1}: $1\in S$ and $n\in S^c$. By successive expansion along rows $i\in S$, we obtain
\begin{equation}
\xi_{Q_S}(\lambda) = (\lambda+2)^{|S|}\xi_B(\lambda)
\label{equ:xicycle}
\end{equation}
where $B$ is a block-diagonal matrix with blocks of the form (\ref{equ:Btridiagonal}) with $c_1=c_2 = 0$ (referred to as type $1$). Since $B$ is a block-diagonal matrix, we have $\xi_B(\lambda) = \prod_{i=1}^b \xi_{B_i}(\lambda)$, $1< b\leq |S^c|$, where $B_i$ is a matrix of type~1. The largest eigenvalue is $\lambda_1$, given in (\ref{equ:ltype1}), for $m = |S^c|$.\\

\noindent \emph{Case~2}: $1\in S^c$ and $n\in S$. In this case, the same arguments hold as in Case~1.

\noindent \emph{Case~3}: $1\in S^c$ and $n\in S^c$. In this case, (\ref{equ:xicycle}) holds, with matrix $B$ with diagonal blocks and other elements as in Case~1, except that $b_{1,|S^c|} = b_{|S^c|,1} = 1$. This matrix can be transformed into a block-diagonal matrix of the same form as in Case~1 by permuting the rows of the matrix, hence, it has the same spectral properties as the matrix $B$ under Case~1. Specifically, this can be done by moving the block of rows that correspond to $B_1$ to the bottom of the matrix $B$.

\noindent \emph{Case~4}: $1\in S$ and $n\in S$. In this case, the same arguments apply as in Case~1.

Finally, we observe that for each of the four cases, since $|S^c|\leq 2(1-\alpha)$, we can take $\delta(Q,\alpha)$ as asserted in the lemma, which completes the proof.

\subsection{Proof of Lemma~\ref{thm:star} (Star)}\label{sec:delta-star} We separately consider the two cases for which either the hub is in the set $S$ or not.
\paragraph{Case~1} Suppose that the hub is in the set $S$, i.e. $1\in S$. It is easy to observe that in this case, the matrix $Q_S$ is a triangular matrix with all upper diagonal elements equal to $0$, and the diagonal elements equal to $(-1,-\frac{1}{n-1},\ldots,-\frac{1}{n-1})$. Hence, the largest eigenvalue is $-\frac{1}{n-1}$.
\paragraph{Case~2} Suppose now that the hub is not in the set $S$, i.e. $1\in S^c$. If $\lambda$ is an eigenvalue of $Q_S$ with an eigenvector $\vec{x}$, then we have
\begin{eqnarray*}
\lambda x_1 &=& - x_1+\frac{1}{n-1}\sum_{i\in S^c\setminus\{1\}} x_i\\
\lambda x_i &=& -\frac{x_i}{n-1}  +\frac{x_1}{n-1},\ \hbox{ for } i\in S^c\setminus\{1\}\\
\lambda x_i &=& -\frac{1}{n-1}x_i, \ \hbox{ for } i \in S.
\end{eqnarray*}
This implies
\begin{eqnarray*}
\lambda x_1 &=& - x_1+\frac{1}{n-1}\sum_{i\in S^c\setminus\{1\}} x_i\\
x_1 &=& ((n-1)\lambda+1) x_i ,\ \hbox{ for } i\in S^c\setminus\{1\}\\
\lambda x_i &=& -\frac{1}{n-1}x_i, \ \hbox{ for } i \in S.
\end{eqnarray*}
Suppose that $\vec{x}$ is such that $x_i = 0$, for every $i\in S$. From the last above identities, it readily follows that $\lambda$ is a solution of the quadratic equation
$$
\lambda^2 + \frac{n}{n-1}\lambda + \frac{1}{n-1}\left(1-\frac{|S^c|-1}{n-1}\right) = 0.
$$
It is straightforward to show that the two solutions are
$$
\lambda_1 = -\frac{1}{2}\frac{n}{n-1}\left(1-\sqrt{1-\frac{4|S|}{n^2}}\right) \hbox{ and } \lambda_2 = -\frac{1}{2}\frac{n}{n-1}\left(1+\sqrt{1-\frac{4|S|}{n^2}}\right).
$$
Clearly, the largest eigenvalue is $\lambda_1$ and since $ |S|\geq (2\alpha-1)n$, it is maximized for $|S|=(2\alpha-1)n$. \\
Finally, we note that the largest eigenvalue is attained in Case~2, which establishes the first equality in the lemma, from which the asserted inequality and its tightness readily follow.

\subsection{Proof of Lemma~\ref{lem:er} (Erd\"os-R\'enyi)}\label{pro:er} From (\ref{equ:quad}), note that for every $S\subset V$ such that $0<|S|<n$, if $\lambda$ is an eigenvalue of matrix $Q_S$, then
$$
\lambda =-\sum_{i\in S^c,j\in S} q_{i,j} x_i^2 -\frac{1}{2}\sum_{i,j\in S^c}q_{i,j} (x_i-x_j)^2\leq -\min_{i\in S^c}\left\{\sum_{j\in S} q_{i,j}\right\}.
$$
since $\sum_{i\in S^c} x_i^2=1$. In the following, we would like to find a value $x_n > 0$ such that $\min_{i\in S^c}\left\{\sum_{j\in S} q_{i,j} \right\} > x_n$ holds with high probability. To this end, we consider the probability that the latter event does not hold, i.e., for $x_n > 0$, we consider
$$
p_e := \P\left(\min_{i\in S^c}\left\{\sum_{j\in S} q_{i,j} \right\} \leq x_n\right).
$$
We first show that the following bound holds $p_e \leq \bar{p}_e$ with
\begin{equation}
\bar{p}_e = 2(1-\alpha)n \exp\left(-(2\alpha-1)(n-1)p_n \varphi\left(\frac{x_n}{2\alpha-1}\right)\right)
\label{equ:barpe}
\end{equation}
where we define $\varphi(x) = x\log(x) + 1 - x$, for $x\geq 0$.\\
To see this, note that for every fixed $\theta > 0$,
\begin{eqnarray*}
p_e =\P\left(\bigcup_{i\in S^c} \left\{\sum_{j\in S} q_{i,j}< x_n\right\}\right)
& \leq & |S^c| \P\left(\sum_{j\in S} q_{i,j}< x_n\right)\\
&\leq & |S^c| e^{\theta x_n} \E\left(e^{-\frac{\theta}{(n-1)p_n}X_{i,j}}\right)^{|S|}\\
&=& |S^c| e^{\theta x_n}\left(1+p_n\left(e^{-\frac{\theta}{(n-1)p_n}}-1\right)\right)^{|S|}
\end{eqnarray*}
where the first inequality follows by the union bound, the second inequality by the Chernoff's inequality, and the third equality by the fact that $X_{i,j}$ is a Bernoulli random variable with mean $p_n$.\\
Since $|S^c|\leq 2(1-\alpha)n$ and $|S|\geq (2\alpha-1)n$, we have
$$
p_e \leq 2(1-\alpha)n e^{\theta x_n}\left(1+p_n\left(e^{-\frac{\theta}{(n-1)p_n}}-1\right)\right)^{(2\alpha-1)n}.
$$
Furthermore, using the fact $1+p_n (e^{-\frac{\theta}{(n-1)p_n}}-1) \leq \exp(p_n (e^{-\frac{\theta}{(n-1)p_n}}-1))$, it follows
$$
p_e \leq 2(1-\alpha)n e^{\theta x_n + (2\alpha-1)np_n \left(e^{-\frac{\theta}{(n-1)p_n}}-1\right)}.
$$
It is straightforward to check that the right-hand side in the last inequality is minimized for $\theta = -(n-1)p_n \log\left(\frac{x_n}{2\alpha - 1}\right)$ and for this value is equal to $\bar{p}_e$ given in (\ref{equ:barpe}).\\
Requiring $\bar{p}_e \leq \frac{1}{n}$ is equivalent to
\begin{eqnarray*}
\varphi\left(\frac{x_n}{2\alpha -1}\right) &\geq & 
\frac{2\log(n) + \log(2(1-\alpha))}
{(2\alpha-1)np_n}.
\end{eqnarray*}
From this it follows that $\lambda \leq -x_n$ holds with high probability provided that $x_n\geq 0$ can be chosen such that
\begin{equation}
\varphi\left(\frac{x_n}{2\alpha -1}\right) \geq 
\frac{2}{c(2\alpha-1)}
\left(
1+\frac{\log(2(1-\alpha))}{2\log(n)}
\right).
\label{equ:suff}
\end{equation}
Such a value $x_n$ exists as $\varphi(x)$ is a decreasing function on $[0,1]$, with boundary values $\varphi(0) = 0$ and $\varphi(1) = 0$, and under our assumption, $c(2\alpha-1) > 2$, the right-hand side in (\ref{equ:suff}) is smaller than $1$ for large enough $n$.\\
Finally, from (\ref{equ:suff}), we note
$$
x_n \geq (2\alpha-1)\varphi^{-1}\left(\frac{2}{c(2\alpha-1)}\right) + O\left(\frac{1}{\log(n)}\right)
$$
from which the asserted result follows. 
\section{Proof of Proposition~\ref{pro:star-phase1}}\label{sec:star-phase1} Let $H(t)$ denote the state of the hub at time $t$. Due to the symmetry of the considered graph, it is not difficult to observe that the dynamics is fully described by a continuous-time Markov process $(H(t), |S_0(t)|, |S_1(t)|, |S_{e_0}(t)|, |S_{e_1}(t)|)_{t\geq 0}$. We need to compute the expected value of the smallest time $t$ such that  $|S_1(t)|=0$, i.e. the time when the state $1$ becomes depleted. To this end, it suffices to consider $(H(t), |S_0(t)|, |S_1(t)|, |S_{e}(t)|)$ where $S_e(t) = S_{e_0}(t) \cup S_{e_1}(t)$, i.e. the system states $e_0$ and $e_1$ are lumped into one state, which we denote with $e$. The system will be said to be in mode $i$ at time $t$, whenever the number of depleted state $1$ nodes before time $t$ is equal to $i$, for $i = 0, 1,\ldots,|S_1(0)|-1$. Notice that if the system is in mode $i$ at time $t$, then $|S_0(t)| = |S_0(0)|-i$, $|S_1(t)|=|S_1(0)|-i$, and $|S_e(t)| = |S_e(0)| + 2i$. For simplicity, we will use the following notation $x_0^i = |S_0(0)|-i$, $x_1^i = |S_1(0)|-i$, and $x_e^i = |S_e(0)| + 2i$, for every given mode $i$.\\
We will compute the expected sojourn time in each of the modes by analyzing a discrete-time Markov chain $\phi^i = (\phi_k^i)_{k\geq 0}$, for given mode $i$, defined as follows. This Markov chain is embedded at time instances at which the hub node interacts with a leaf node. The state space of $\phi^i$ consists of the states $0$, $1$, $e$, $e^*$ with the transition probabilities given in Figure~\ref{fig:trans}. The state of the Markov chain $\phi^i$ indicates the state of the hub at contact instances of the hub with the leaf nodes, where we introduced an extra state $e^*$ to encode the event where the hub is in state $e$ and that this state was reached by the hub from either state $0$ or $1$, thus indicating a depletion of state $1$. Note that the expected duration of mode $0$ is equal to the mean hitting time of state $e^*$ for the Markov chain $\phi^0$ started at either state $0$ or $1$, while the expected duration of mode $i$, for $0 < i < |S_1(0)|$, is equal to the mean hitting time of state $e^*$ for the Markov chain $\phi^i$ started at state $e$. We compute these mean hitting times in the following.
\begin{figure}[t]
\vspace*{0cm}
\begin{center}
\psfig{figure=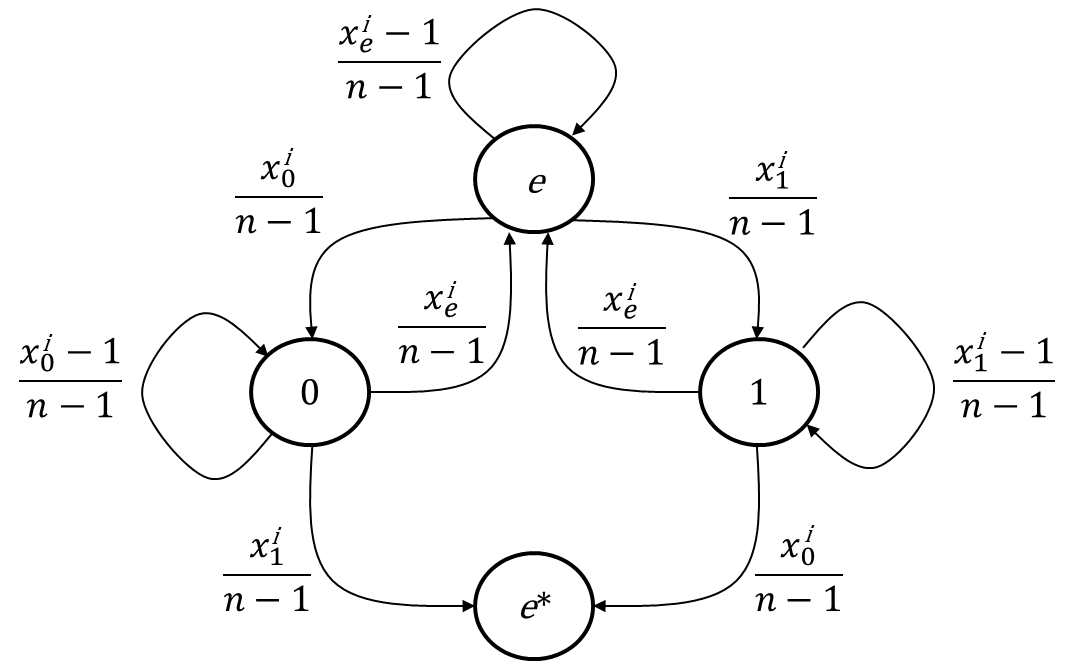,height=2.5in,width=3.5in}
\caption{Star-shaped network: the transition probabilities of $\phi^i$.}
\label{fig:trans}
\end{center} 
\end{figure}
Fix an arbitrary mode $i$, and then let $\varphi_s(i)$ be the mean hitting time of state $e^*$ for the Markov chain $\phi^i$ started at state $s=0$, $1$, and $e$. By the first-step analysis, we have that the latter mean hitting times are the solution of the following system of linear equations
\begin{equation}
\begin{array}{rl}
\varphi_0(i) &= \frac{x_0^i-1}{n-1}\varphi_0(i) +\frac{x_e^i}{n-1}\varphi_e(i) + 1\\ 
\varphi_e(i) &= \frac{x_0^i}{n-1}\varphi_0(i) + \frac{x_e^i-1}{n-1}\varphi_e(i) + \frac{x_1^i}{n-1}\varphi_1(i) + 1\\
\varphi_1(i) &=  \frac{x_e^i}{n-1}\varphi_e(i) + \frac{x_1^i-1}{n-1}\varphi_1(i) +  1. 
\end{array}
\label{equ:lin}
\end{equation}
From this, it is straightforward to derive
\begin{eqnarray}
\varphi_0(i) &=& (n-1)\frac{n^2 x_e^i + x_0^i x_1^i}{x_0^i x_1^i (n-x_0^i)(n+x_e^i)},\label{equ:m0}\\ 
\varphi_1(i) &=& (n-1)\frac{n^2 x_e^i + x_0^i x_1^i}{x_0^i x_1^i (n-x_1^i)(n+x_e^i)},\label{equ:m1}\\
\varphi_e(i) &=& (n-1)\frac{n^2 - x_0^i x_1^i}{x_0^i x_1^i(n+x_e^i)}.\label{equ:me}
\end{eqnarray}
The expected sojourn time in each given mode is as follows. For mode $i = 0$, it holds $x_0^0 = \alpha n$, $x_1^0 = (1-\alpha)n$, $x_e^0 = 0$, and thus, from (\ref{equ:m0}) and (\ref{equ:m1}), the expected sojourn in mode $0$ is $\varphi_0(0)$ and $\varphi_1(0)$, for the initial state of the hub equal to $0$ and $1$, respectively, where 
\begin{equation}
\varphi_0(0) = \frac{n-1}{(1-\alpha)n} \hbox{ and } \varphi_1(0) = \frac{n-1}{\alpha n}.
\label{equ:m0m1}
\end{equation}
On the other hand, for $0<i<|S_1(0)|$, the expected sojourn time in mode $i$ is $\varphi_e(i)$, and from (\ref{equ:me}) and $x_0^i = |S_0(0)|-i$, $x_1^i = |S_1(0)|-i$, $x_e^i = |S_e(0)| + 2i$, we have 
\begin{equation}
\varphi_e(i) = (n-1)\left(\frac{n^2}{(\alpha n - i)((1-\alpha)n-i)(n+2i)}-\frac{1}{n+2i}\right).
\label{equ:mei}
\end{equation}
Finally, the expected duration of phase $1$ is equal to $\varphi_s(0) + \sum_{i=1}^{(1-\alpha)n-1} \varphi_e(i)$, where $s$ denotes the initial state of the hub, either $0$ or $1$. On the one hand, from (\ref{equ:m0m1}), we have that for every fixed $\alpha \in (1/2,1]$, both $\varphi_0(0)$ and $\varphi_1(0)$ are asymptotically constants, as the number of nodes $n$ grows large, thus $\varphi_0(0) = \Theta(1)$ and $\varphi_1(0) = \Theta(1)$. On the other hand, using (\ref{equ:mei}) and some elementary calculus, we obtain
\begin{eqnarray*}
 \sum_{i=1}^{(1-\alpha)n-1} \varphi_e(i)
&=& \frac{n-1}{(2\alpha-1)(3-2\alpha)}H_{(1-\alpha)n-1}  - \frac{n-1}{(2\alpha-1)(1+2\alpha)}
[H_{\alpha n-1}-H_{(2\alpha-1)n}]\\
&& + \left(\frac{2}{(2\alpha-1)(3-2\alpha)} - \frac{2}{(2\alpha-1)(1+2\alpha)}-1\right)
\sum_{i=1}^{(1-\alpha)n-1}\frac{n-1}{n+2i}.
\end{eqnarray*}
where, recall, $H_k=\sum_{i=1}^k \frac{1}{i}$. From this, it can be observed that
\begin{equation}
\sum_{i=1}^{(1-\alpha)n-1} \varphi_e(i) = \frac{1}{(2\alpha-1)(3-2\alpha)} n \log(n) + O(n).
\label{equmas}
\end{equation}
This completes the proof of the proposition.
\end{document}